\documentclass[11pt]{amsart}
\usepackage{amsfonts}
\usepackage{mathrsfs}
\usepackage{amsmath}
\usepackage{amsmath, amsthm, amssymb}
\usepackage{color}

\renewcommand{\(}{\left( }
\renewcommand{\)}{\right) }

\renewcommand{\theequation}{\theequation. \arabic{equation}}
\numberwithin{equation}{section}
\newtheorem{thm}{Theorem}[section]

\newtheorem{lem}[thm]{Lemma}
\newtheorem{rem}[thm]{Remark}
\newtheorem{prop}[thm]{Proposition}
\newtheorem{defn}[thm]{Definition}

\def\squarebox#1{\hbox to #1{\hfill\vbox to #1{\vfill}}}

\begin{document}
\title{On the $q$-partial differential equations and $q$-series}
\author{Zhi-Guo Liu}
\footnote{Dedicated to Srinivasa Ramanujan on the occasion of his 125th birth anniversary}
\date{\today}
\address{East China Normal University, Department of Mathematics, 500 Dongchuan Road,
Shanghai 200241, P. R. China} \email{zgliu@math.ecnu.edu.cn;
liuzg@hotmail.com}
\begin{abstract}
Using the theory of functions of several complex variables, we prove that if an analytic
function in several variables satisfies a system of $q$-partial differential equations,
then,  it can be expanded in terms of the product of the Rogers-Szeg\H{o} polynomials.
This expansion theorem allows us to develop a general method for proving $q$-identities.
A general $q$-transformation formula is derived, which implies Watson's $q$-analog of Whipple's theorem
as a special case. A multilinear generating function for the Rogers-Szeg\H{o} polynomials is given.
The theory of $q$-exponential operator is revisited.
\end{abstract}
\maketitle
\tableofcontents
%%%%%%%%%%%%%%%%%%%%%%%%%%%%%%%%%%%%%%%%%%%%%%%%%%%%%%%%%%%%%%%%%%%%%%%%%%%%%%%%%%%%%
\section{Introduction}
%%%%%%%%%%%%%%%%%%%%%%%%%%%%%%%%%%%%%%%%%%%%%%%%%%%%%%%%%%%%%%%%%%%%%%%%%%%%%%%%%%%%%%%
Throughout the paper, we use the standard $q$-notations.
For $0<q<1$, we define the $q$-shifted factorials as
\begin{equation*}
(a; q)_0=1,\quad (a; q)_n=\prod_{k=0}^{n-1}(1-aq^k), \quad (a;
q)_\infty=\prod_{k=0}^\infty (1-aq^k);
\end{equation*}
and for convenience, we also adopt the following compact notation for the multiple
$q$-shifted factorial:
\begin{equation*}
(a_1, a_2,...,a_m;q)_n=(a_1;q)_n(a_2;q)_n ... (a_m;q)_n,
\end{equation*}
where $n$ is an integer or $\infty$.

The basic hypergeometric series
${_r\phi_s}$ is defined as
\begin{equation*}
{_r\phi_s} \left({{a_1, a_2, ..., a_{r}} \atop {b_1, b_2, ...,
b_s}} ;  q, z  \right) =\sum_{n=0}^\infty \frac{(a_1, a_2, ...,
a_{r};q)_n} {(q,  b_1, b_2, ..., b_s ;q)_n}\left((-1)^n q^{n(n-1)/2}\right)^{1+s-r} z^n.
\end{equation*}

For any function $f(x)$ of one variable, the  $q$-derivative of $f(x)$
with respect to $x,$ is defined as
\begin{equation*}
\mathcal{D}_{q,x}\{f(x)\}=\frac{f(x)-f(qx)}{x},
\end{equation*}
and we further define  $\mathcal{D}_{q,x}^{0} \{f\}=f,$ and for $n\ge 1$, $\mathcal{D}_{q, x}^n \{f\}=\mathcal{D}_{q, x}\{\mathcal{D}_{q, x}^{n-1}\{f\}\}.$

For any nonnegative integer $n,$ we have the higher-order $q$-derivative formula
\begin{equation}
\mathcal{D}^n_{q, x} \left\{\frac{1}{(sx; q)_\infty}\right\}=\frac{s^n}{(sx; q)_\infty},
\label{qd:eqn1}
\end{equation}
which is the case $t=0$ of the following  general higher-order $q$-derivative formula:
\begin{equation}
\mathcal{D}^n_{q, x} \left\{\frac{(tx; q)_\infty}{(sx; q)_\infty}\right\}
=s^n(t/s; q)_n \frac{(q^ntx; q)_\infty}{(sx; q)_\infty}.
\label{qd:eqn2}
\end{equation}
\begin{defn}\label{qpde}
A $q$-partial derivative of a function of several variables is its $q$-derivative with respect to one of those variables, regarding other variables as constants. The $q$-partial derivative of a function $f$ with respect to the variable $x$ is denoted by $\partial_{q, x}\{f\}$.
\end{defn}
The Gaussian binomial coefficients also called the $q$-binomial coefficients are $q$-analogs of the binomial coefficients, which
are given by
\begin{equation}
{n\brack k}_q=\frac{(q; q)_n}{(q; q)_k(q; q)_{n-k}}.
\label{qd:eqn3}
\end{equation}
Now we introduce the definition of the Rogers--Szeg\H{o} polynomials which were first studied by Rogers \cite{Rogers1893} and then by Szeg\H{o} \cite{Szeg}.
\begin{defn} \label{rspolydefn}With the $q$-binomial coefficients be defined as in (\ref{qd:eqn3}),
the Rogers-Szeg\H{o} polynomials are defined by
\[
h_n(x, y|q)=\sum_{k=0}^n {n\brack k}_q x^k y^{n-k}.
\]
\end{defn}

If $q$ is replaced by $q^{-1}$ in the Rogers-Szeg\H{o} polynomials, we can obtain the Stieltjes-Wigert polynomials
(see, for example,  \cite{Carlitz, Szeg}).
\begin{defn}\label{swpolydefn}  The Stieltjes-Wigert polynomials are defined by
\[
g_n(x, y|q)=h_n(x, y|q^{-1})=\sum_{k=0}^n {n\brack k}_q  q^{k(k-n)}x^k y^{n-k}.
\]
\end{defn}
\begin{prop}\label{gefunpp} $h_n(x, y|q)$ and $g_n(x, y|q)$ satisfy the identities
\begin{equation}
\partial_{q, x} \{h_n(x, y|q)\}=\partial_{q, y} \{h_n(x, y|q)\}=(1-q^n)h_{n-1}(x, y|q),
\label{rseqn3}
\end{equation}
\begin{equation}
\partial_{q^{-1}, x} \{g_n(x, y|q)\}=\partial_{q^{-1}, y} \{g_n(x, y|q)\}=(1-q^{-n})g_{n-1}(x, y|q).
\label{rseqn4}
\end{equation}
\end{prop}
\begin{proof}
Using the identity, $\partial_{q, x} \{ x^k\}=(1-q^k) x^{k-1},$  we immediately find that
\[
\partial_{q, x}\left\{h_n(x, y|q)\right\}
=\sum_{k=1}^n {n\brack k}_q  (1-q^k) x^{k-1} y^{n-k}.
\]
In the same way, using the identity, $\partial_{q, y} \{ y^{n-k}\}=(1-q^{n-k}) y^{n-k-1},$ we deduce that
\[
\partial_{q, y}\{h_n(x, y|q)\}
=\sum_{k=0}^{n-1} {n\brack k}_q (1-q^{n-k}) x^{k} y^{n-k-1}.
\]
If we make the variable change $k+1 \to k$ in the right-hand side of the above equation, we can find that
\[
\partial_{q, y}\left\{h_n(x, y|q)\right\}
=\sum_{k=1}^{n} {n\brack {k-1}}_q (\alpha; q)_{k} (1-q^{n-k+1}) x^{k-1} y^{n-k}.
\]
From the definition of the $q$-binomial coefficients, it is easy to verify that
\[
{n\brack k}_q (1-q^k)={n\brack {k-1}}_q  (1-q^{n-k+1}).
\]
Thus, the identity in (\ref{rseqn3}) holds. In this same way, we can prove (\ref{rseqn4}).
This completes the proof of Proposition~\ref{gefunpp}.
\end{proof}
To explain our motivation of this paper, we begin with the following proposition.
\begin{prop} \label{ppmotivation}
If $f(x, y)$ is a two variables analytic function in a neighbourhood
of $(0, 0) \in \mathbb{C}^2$, satisfying the partial differential equation
$f_x(x,y)=f_y(x, y),$ then,  we have $f(x, y)=f(x+y, 0).$
\end{prop}
\begin{proof}
Now we begin to solve the partial differential equation in the above proposition.
From the theory of two complex variables, we may assume that near $(x, y)=(0, 0),$
\[
f(x, y)=\sum_{n=0}^\infty A_n(x)y^n.
\]
If this is substituted into $ f_x(x,y)=f_y(x, y)$, we immediately conclude that
\[
\sum_{n=0}^\infty A_n'(x)y^n=\sum_{n=0}^\infty (n+1)A_{n+1}(x)y^n.
\]
Equating the  coefficients of $y^n$ on both sides of the above equation,
we find that for each integer $n\ge 1, A_n(x)={A_{n-1}'(x)}/{n}.$
By iteration, we deduce that $A_n(x)={A_0^{(n)}(x)}/{n!}.$  It is obvious that
$A_0(x)=f(x, 0).$ Using the Taylor expansion, we deduce that
\[
f(x, y)=\sum_{n=0}^\infty \frac{f^{(n)}(x, 0)}{n!}y^n=f(x+y, 0),
\]
which completes the proof of Proposition~\ref{ppmotivation}.
\end{proof}
In order to find the $q$-extension of Proposition~\ref{ppmotivation}, we are led to the following proposition.
\begin{prop}\label{qppmotivation} If $f(x,y)$  is a two-variable
 analytic function at $(0,0)\in \mathbb{C}^2$, then, we have
 \begin{itemize}
 \item[(i)]
 $f$ can be expanded in terms of $h_n(x, y|q)$ if and only if $f$
 satisfies the $q$-partial differential equation
 $
 \partial_{q, x}\{f\}=\partial_{q, y}\{f\}.
 $
\item[(ii)]
 $f$ can be expanded in terms of $g_n(x, y|q)$  if and only if  $f$
 satisfies the $q$-partial differential equation
$
\partial_{q^{-1}, x}\{f\}=\partial_{q^{-1}, y}\{f\}.
$
\end{itemize}
\end{prop}
This proposition can be extended to the following more general expansion  theorem
for the  analytic functions in several variables, which is the main
result of this paper.
\begin{thm}\label{mainthmliu}
If $f(x_1,y_1, \ldots, x_k, y_k)$  is a $2k$-variable
 analytic function at $(0,0, \cdots, 0)\in \mathbb{C}^{2k}$, then,  we have
\begin{itemize}
 \item[(i)]
 $f$ can be expanded in terms of $h_{n_1}(x_1, y_1|q)\cdots h_{n_k}(x_k, y_k|q)$  if and only if  $f$
 satisfies the $q$-partial differential equations
 $
 \partial_{q, x_j}\{f\}=\partial_{q, y_j}\{f\}~\text{for}~j=1, 2, \ldots,  k.
 $
\item[(ii)]
 $f$ can be expanded in terms of $g_{n_1}(x_1, y_1|q)\cdots g_{n_k}(x_k, y_k|q)$  if and only if $f$
 satisfies the $q$-partial differential equation
$
\partial_{ q^{-1}, x_j}\{f\}=\partial_{q^{-1}, y_j}\{f\}~\text{for}~j=1, 2, \ldots,  k.
$\end{itemize}
\end{thm}
Proposition~\ref{qppmotivation} is the special case $k=1$ of Theorem~\ref{mainthmliu}.
This theorem is useful in $q$-series, which allows us to develop a general method for proving $q$-identities.
Many applications of this expansion theorem to $q$-series are discussed in this paper.

To determine if a given function is an analytic functions in several complex variables,
we often use the following theorem (see, for example, \cite[p. 28]{Taylor}).
\begin{thm}\label{hartogthm} {\rm (Hartog's theorem).}
If a complex valued function $f(z_1, z_2, \ldots, z_n)$ is holomorphic (analytic) in each variable separately in a domain $U\in\mathbb{C}^n,$
then,  it is holomorphic (analytic) in $U.$
\end{thm}

%/////////////////////////////////////////////////////////////////////
%%%%%%%%%%%%%%%%%%%%%%%%%%%%%%%%%%%%%%%%%%%%%%%%%%%%%%%%%%%%%%%%%%%%%%%%%%%%%%%%%%%%%
\section{Proof of the expansion theorem}
%%%%%%%%%%%%%%%%%%%%%%%%%%%%%%%%%%%%%%%%%%%%%%%%%%%%%%%%%%%%%%%%%%%%%%%%%%%%%%%%%%%%%%%
In order to prove Theorem~\ref{mainthmliu},  we need the following fundamental property of
several complex variables (see, for example, \cite[p. 5, Proposition~ 1]{Malgrange}, \cite[p. 90]{Range}).
\begin{prop}\label{mcomplexpp}
If $f(x_1, x_2, \ldots,  x_k)$ is analytic at the origin $(0, 0, \ldots,  0)\in \mathbb{C}^k$, then,
$f$ can be expanded in an absolutely convergent power series,
 \[
 f(x_1, x_2, \ldots,  x_k)=\sum_{n_1, n_2, \ldots, n_k=0}^\infty \alpha_{n_1, n_2, \ldots, n_k}
 x_1^{n_1} x_2^{n_2}\cdots x_k^{n_k}.
 \]
\end{prop}
Now we begin to prove Proposition~\ref{qppmotivation} using Proposition~\ref{mcomplexpp}.
\begin{proof}
The proofs of (i) and (ii) are similar, so we only prove (i).  Since $f$ is analytic at $(0, 0),$
 from Proposition~\ref{mcomplexpp}, we know that $f$ can be expanded in
an absolutely convergent power series in a neighborhood of $(0, 0)$.
Thus there exists a sequence $\{\alpha_{m, n}\}$ independent of $x$ and $y$ such that
\begin{equation}
 f(x, y)=\sum_{m, n}^\infty \alpha_{m, n} x^my^n=\sum_{n=0}^\infty y^n \left\{\sum_{m=0}^\infty \alpha_{m, n} x^m\right\}.
 \label{pthm:eqn1}
\end{equation}
Substituting this into the $q$-partial differential equation $\partial_{q, x}\{f(x, y)\}=\partial_{q, y}\{f(x, y)\}$
and using the fact $\partial_{q, y}\{y^n\}=(1-q^n)y^{n-1},$ we find that
\[
\sum_{n=0}^\infty y^n \partial_{q, x}\left\{\sum_{m=0}^\infty \alpha_{m, n} x^m\right\}
=\sum_{n=1}^\infty (1-q^n)y^{n-1}\left\{\sum_{m=0}^\infty \alpha_{m, n} x^m\right\}.
\]
Equating the coefficients of $y^{n-1}$ on both sides of the above equation, we easily deduce that
\[
\sum_{m=0}^\infty \alpha_{m, n} x^m=\frac{1}{1-q^n} \partial_{q, x}\left\{\sum_{m=0}^\infty \alpha_{m, n-1} x^m\right\}.
\]
Iterating the above equation $(n-1)$ times, we conclude that
\[
\sum_{m=0}^\infty \alpha_{m, n} x^m=\frac{1}{(q; q)_n} \partial^n_{q, x}\left\{\sum_{m=0}^\infty \alpha_{m, 0} x^m\right\}.
\]
With the help of  the identity,  ${\partial}^n_{q, x}\{x^m\}=(q; q)_m x^{m-n}/(q; q)_{m-n}$, we obtain
\[
\sum_{m=0}^\infty \alpha_{m, n} x^m=\sum_{m=n}^\infty \alpha_{m, 0}{m\brack n}_q x^{m-n}.
\]
Noting that the series in (\ref{pthm:eqn1}) is absolutely convergent, substituting the above equation
into (\ref{pthm:eqn1}) and interchanging the order of the summation, we deduce that
\[
f(x, y)=\sum_{m=0}^\infty \alpha_{m, 0} \sum_{n=0}^m {m \brack n}_qy^n x^{m-n}
=\sum_{m=0}^\infty \alpha_{m, 0}h_m(x, y|q),
\]
since $h_n(x, y|q)$ is symmetric in $x$ and $y$. Conversely, if $f(x, y)$ can be expanded in terms of $h_n(x, y|q), $
then using (\ref{rseqn3}),  we find that $\partial_{q, x}\{f(x, y)\}=\partial_{q, y}\{f(x, y)\}.$
This completes the proof of Proposition~\ref{qppmotivation}.
\end{proof}
\begin{rem}{\rm It can be shown that Proposition~\ref{qppmotivation} is equivalent to \cite[Theorem~4]{Liu2010}, but the proof in \cite{Liu2010}
is less rigorous. This paper may be viewed as an improved of improved version of \cite{Liu2010}.}
\end{rem}
Now we begin to prove Theorem~\ref{mainthmliu} by using Proposition~\ref{qppmotivation} and mathematical induction.
\begin{proof} The proof of (ii) is similar to that of (i), so we only prove (i). From  Proposition~\ref{qppmotivation}
we conclude that the theorem holds when $k=1.$ Now, we assume that the theorem is true for the case $k-1$ and consider the case $k$.
If we regard $f(x_1, y_1, \ldots, x_k, y_k)$ as a function of $x_1$ and $y_1, $ then $f$ is analytic at $(0, 0)$ and satisfies
$\partial_{q, x_1}\{f\}=\partial_{q, y_1}\{f\}.$ Thus from (i) in Proposition~\ref{qppmotivation}, there exists a sequence
$\{c_{n_1}(x_2, y_2, \ldots, x_k, y_k)\}$ independent of $x_1$ and $y_1$ such that
\begin{equation}
f(x_1, y_1, \ldots, x_k, y_k)=\sum_{n_1=0}^\infty c_{n_1}(x_2, y_2, \ldots, x_k, y_k)h_{n_1}(x_1, y_1|q).
\label{pthm:eqn2}
\end{equation}
Setting $y_1=0$ in the above equation and using $h_{n_1}(x_1, 0|q)=x_1^{n_1},$ we obtain
\[
f(x_1, 0, x_2, y_2,  \ldots, x_k, y_k)=\sum_{n_1=0}^\infty c_{n_1}(x_2, y_2, \ldots, x_k, y_k)x_1^{n_1}.
\]
Using the Maclaurin expansion theorem, we immediately deduce that
\[
c_{n_1}(x_2, y_2, \ldots, x_k, y_k)=\frac{\partial^{n_1} f(x_1, 0, x_2, y_2,  \ldots, x_k, y_k)}{{n_1}!\partial {x_1}^{n_1}}\Big|_{x_1=0}
\]
Since $f(x_1, y_1, \ldots, x_k, y_k)$ is analytic near $(x_1, y_1, \ldots, x_k, y_k)=(0, \ldots, 0)\in \mathbb{C}^{2k},$ from
the above equation, we know that $c_{n_1}(x_2, y_2, \ldots, x_k, y_k)$ is analytic near $(x_2, y_2, \ldots, x_k, y_k)=(0, \ldots, 0)\in \mathbb{C}^{2k-2}.$ Combining (\ref{pthm:eqn2}) with (i) in Theorem~\ref{mainthmliu}, we find, for $j=2, \ldots k$, that
\begin{align*}
&\sum_{n_1=0}^\infty \partial_{q, x_j}\{c_{n_1}(x_2, y_2, \ldots, x_k, y_k)\} h_{n_1}(x_1, y_1|q)\\
&=\sum_{n_1=0}^\infty \partial_{q, y_j}\{c_{n_1}(x_2, y_2, \ldots, x_k, y_k)\} h_{n_1}(x_1, y_1|q).
\end{align*}
By equating the coefficients of $h_{n_1}(x_1, y_1|q)$ in the above equation, we find that
for $j=2, \ldots,  k,$
\[
\partial_{q, x_j}\{c_{n_1}(x_2, y_2, \ldots, x_k, y_k)\}
=\partial_{q, y_j}\{c_{n_1}(x_2, y_2, \ldots, x_k, y_k)\}.
\]
Thus by the inductive hypothesis, there exists a sequence $\{\alpha_{n_1, n_2, \ldots, n_k}\}$ independent of
$x_2, y_2, \ldots, x_k, y_k$ (of course independent of $x_1$ and $y_1$) such that
\[
c_{n_1}(x_2, y_2, \ldots, x_k, y_k)=\sum_{n_2, \ldots, n_k=0}^\infty \alpha_{n_1, n_2, \ldots, n_k}
h_{n_2}(x_2, y_2|q)\ldots h_{n_k}(x_k, y_k|q).
\]
Substituting this equation into (\ref{pthm:eqn2}), we find that $f$ can be expanded
in terms of $h_{n_1}(x_1, y_1|q)\cdots h_{n_k}(x_k, y_k|q).$

Conversely, if $f$ can be expanded in terms of $h_{n_1}(x_1, y_1|q)\cdots h_{n_k}(x_k, y_k|q),$
then using (\ref{rseqn3}),  we find that $\partial_{q, x_j}\{f\}=\partial_{q, y_j}\{f\}$ for
$j=1, 2, \ldots, k.$ This completes the proof of Theorem~\ref{mainthmliu}.
\end{proof}
%/////////////////////////////////////////////////////////////////////
%/////////////////////////////////////////////////////////////////////
\section{The generating functions of the Rogers-Szeg\H{o} polynomials and the Stieltjes-Wigert polynomials}
%/////////////////////////////////////////////////////////////////////
\begin{thm} \label{gefunthm}  If $h_n(x, y|q)$ and $g_n(x, y|q)$ are given by
Definitions~\ref{rspolydefn} and \ref{swpolydefn}, then,  we have
\begin{equation}
\sum_{n=0}^\infty h_n (x, y | q) \frac{t^n}{(q; q)_n}=\frac{1}{(xt, yt; q)_\infty},~\max\{|xt|, |yt|\}<1,
\label{rseqn1}
\end{equation}
 \begin{equation}
\sum_{n=0}^\infty (-1)^n q^{n(n-1)/2} g_n(x, y|q)
\frac{t^n}{(q; q)_n} =(xt, yt; q)_\infty. \label{rseqn2}
\end{equation}
\end{thm}
\begin{proof} We only prove (\ref{rseqn1}). The proof of (\ref{rseqn2}) is similar, so is omitted.
It is well-known that $1/(xt; q)_\infty$ is an analytic function of $x$ for $|xt|<1,$ and
 $1/(yt; q)_\infty$ is an analytic function of $y$ for $|yt|<1.$ Thus, $1/(xt, yt; q)_\infty$
 is an analytic function of $x$ and $y$ for $\max\{|xt|, |yt|\}<1$. If we use $f(x, y)$ to denote
 the right-hand side of (\ref{rseqn1}), then $f(x, y)$ is analytic near $(0, 0)\in \mathbb{C}^2.$
 A direct computation shows that
 \[
 \partial_{q, x}\{f(x, y)\}= \partial_{q, y}\{f(x, y)\}=tf(x, y).
 \]
 Thus by (i) in Proposition~\ref{qppmotivation}, there exists a sequence $\{\alpha_n\}$ independent of $x$ and $y$
such that
\[
f(x, y)=\frac{1}{(xt, yt; q)_\infty}=\sum_{n=0}^\infty \alpha_n h_n(x, y|q).
\]
Taking $y=0$ in the above equation,   using $h_n(x, 0|q)=x^n$ and the $q$-binomial theorem,  we obtain
\[
\frac{1}{(xt; q)_\infty}=\sum_{n=0}^\infty \frac{(xt)^n}{(q; q)_n}=\sum_{n=0}^\infty \alpha_n x^n.
\]
Equating the coefficients of $x^n$ on both sides of the above equation, we deduce that
$\alpha_n=t^n/(q; q)_n.$ Thus, we arrive at (\ref{rseqn1}).
\end{proof}
We can also prove Theorem~\ref{gefunthm} by multiplying two copies of $q$-binomial theorem together.
%/////////////////////////////////////////////////////////////////////
\section{$q$-Mehler formulas for the Rogers-Szeg\H{o} polynomials and the Stieltjes-Wigert polynomials}
%/////////////////////////////////////////////////////////////////////
Using Theorem~\ref{mainthmliu},  we can derive easily the $q$-Mehler formulas for the Rogers-Szeg\H{o} polynomials and
the Stieltjes-Wigert polynomials.
\begin{thm}\label{aqmehler} If $\max\{|xut|, |xvt|, |yut|, |yvt|\}<1$, then, we have
\begin{equation*}
\sum_{n=0}^\infty h_n (x, y|q) h_n (u, v|q) \frac {t^n}{(q; q)_n}
=\frac{(xyuvt^2; q)_\infty} {(xut, xvt, yut, yvt; q)_\infty}.
\end{equation*}
\end{thm}
\begin{thm}\label{bqmehler} If  $|xyuvt^2/q|<1$, then, we have the identity
 \begin{equation*}
 \sum_{n=0}^\infty(-1)^n  g_n (x, y|q) g_n (u, v|q) \frac {q^{n(n-1)/2}t^n}{(q; q)_n}
 =\frac {(xut, xvt, yut, yvt; q)_\infty}{(xyuvt^2/q; q)_\infty}.
 \end{equation*}
 \end{thm}
The $q$-Mehler formula for the Rogers-Szeg\H{o}  polynomials  was first given by
Rogers~\cite{Rogers1893} in 1893 and later reproved by Carlitz~\cite{Carlitz}.
The $q$-Mehler formula  for $g_n(x, y|q)$ was first proved by L. Carlitz \cite{Carlitz}.

\begin{proof} We first prove Theorem~\ref{aqmehler}. If we use $f(x, y)$ to denote the right-hand side of
the equation in Theorem~\ref{aqmehler}, it is obvious that $f(x, y)$ is analytic in $x$ and $y$ separately,
so by Hartog's theorem, we know that $f(x, y)$ is analytic at $(0, 0)$.  Using
the identity $(z; q)_\infty=(1-z)(qz; q)_\infty$ and a direct computation, we find that
\[
\partial_{q, x}\{f(x, y)\}=\partial_{q, y}\{f(x, y)\}
=\frac{t(u+v)-xuvt^2-yuvt^2}{1-xyuvt^2} f(x, y).
\]
Thus by (i) in Proposition~\ref{qppmotivation}, there exists a sequence $\{\alpha_n\}$ independent of $x$ and $y$
such that
\begin{equation}
\frac{(xyuvt^2; q)_\infty} {(xut, xvt, yut, yvt; q)_\infty}=\sum_{n=0}^\infty
\alpha_n h_n(x, y|q).
\label{meqn1}
\end{equation}
Putting $y=0$ in the above equation, using the fact $h_n(x, 0|q)=x^n,$
and the generating function for  $h_n$ in (\ref{rseqn1}),  we find that
\[
\sum_{n=0}^\infty \alpha_n x^n=\frac{1}{(xut, xvt; q)_\infty}
=\sum_{n=0}^\infty h_n(u, v|q)\frac{(xt)^n}{(q; q)_n}.
\]
Equating the coefficients of $x^n$ on both sides of the above equation, we deduce that
$\alpha_n=h_n(u, v) t^n/(q; q)_n.$ Substituting this into (\ref{meqn1}), we complete
the proof of Theorem~\ref{aqmehler}.

Now we turn to prove Theorem~\ref{bqmehler}. Denote the right-hand side of the equation in
Theorem~\ref{bqmehler} by $g(x, y)$, then,  by
the identity,  $(z; q)_\infty=(1-z)(qz; q)_\infty$ and a direct computation, we deduce that
\[
\partial_{q^{-1}, x}\{g(x, y)\}=\partial_{q^{-1}, y}\{g(x, y)\}
=\frac{qt(u+v)-(x+y)uvt^2}{q^2-xyuvt^2} g(x, y).
\]
Thus by  (ii) in Proposition~\ref{qppmotivation}, there exists a sequence $\{\beta_n\}$ independent of $x$ and $y$
such that
\begin{equation}
\frac {(xut, xvt, yut, yvt; q)_\infty}{(xyuvt^2/q; q)_\infty}=\sum_{n=0}^\infty \beta_n g_n(x, y|q).
\label{meqn2}
\end{equation}
Setting $y=0$ in the above equation,   using the fact $g_n(x, 0|q)=x^n,$
and the generating function for  $g_n$ in (\ref{rseqn2}),  we find that
\[
\sum_{n=0}^\infty \beta_n x^n =(xut, xvt; q)_\infty
=\sum_{n=0}^\infty (-1)^n q^{n(n-1)/2} g_n(u, v|q)\frac{(xt)^n}{(q; q)_n}.
\]
It follows that $\beta_n=(-1)^n q^{n(n-1)/2}g_n(u, v|q) t^n/(q; q)_n.$
Substituting this into (\ref{meqn2}), we complete the proof of Theorem~\ref{bqmehler}.
\end{proof}
%%%%%%%%%%%%%%%%%%%%%%%%%%%%%%%%%%%%%%%%%%%%%%%%%%%%%%%%%%%%%%%%%%%%%%%%%%%%%%%%%%%%%%%%%%%%%%%%%%%%
\section{Carlitz's extension of the $q$-Mehler formula for the Rogers-Szeg\H{o} polynomials }
%%%%%%%%%%%%%%%%%%%%%%%%%%%%%%%%%%%%%%%%%%%%%%%%%%%%%%%%%%%%%%%%%%%%%%%%%%%%%%%%%%%%%%%%%%%%%%%%%%%%
Carlitz's extension of the $q$-Mehler formula for the Rogers-Szeg\H{o} polynomials \cite[p. 96, Eq. (4.1)]{Carlitz1972}
is equivalent to the following theorem.  In this section we will prove it using Proposition~\ref{qppmotivation}.
\begin{thm} \label{Carlitzthm} For $\max\{|aut|, |but|, |avt|, |bvt|\}<1,$ we have that
\begin{align*}
\sum_{n=0}^\infty h_{n+k}(a, b|q) h_n(u, v|q) \frac{t^n}{(q; q)_n}
&=\frac{(abuvt^2; q)_\infty}{(aut, but, avt, bvt; q)_\infty}\\
&\times \sum_{j=0}^k {k \brack j}_q \frac{b^j a^{k-j}(aut, avt; q)_j}{(abuvt^2; q)_j}.
\end{align*}
\end{thm}
\begin{proof} Differentiating $k$ times the generating function for $h_n(a, b|q)$ with respect to $t$, we have that
\begin{equation}
\sum_{n=0}^\infty h_{n+k}(a, b|q)  \frac{t^n}{(q; q)_n}
=\frac{1}{(at, bt; q)_\infty}
\sum_{j=0}^k {k \brack j}_q b^j a^{k-j}(at; q)_j.
\label{careqn1}
\end{equation}
If $t$ is replaced by $tu,$ we arrive at
\begin{equation}
\sum_{n=0}^\infty h_{n+k}(a, b|q)  \frac{u^nt^n}{(q; q)_n}
=\frac{1}{(aut, but; q)_\infty}
\sum_{j=0}^k {k \brack j}_q b^j a^{k-j}(aut; q)_j.
\label{careqn2}
\end{equation}
Denote the right hand side of the equation in Theorem \ref{Carlitzthm} by $f(u, v).$ Then by a direct computation,
we find that
\begin{align*}
&\partial_{q, u}\{f(u, v)\}=\partial_{q, v}\{f(u, v)\}\\
&=\frac{(abuvt^2; q)_\infty}{(aut, but, avt, bvt; q)_\infty}
\sum_{j=0}^k {k \brack j}_q \frac{b^j a^{k-j}(aut, avt; q)_j}{(abuvt^2; q)_j}
\(\frac{atq^j+bt-ab(u+v)t^2q^j}{1-abuvt^2q^j}\).
\end{align*}
Thus,  by (i) in Proposition~\ref{qppmotivation}, there exists a sequence $\{\alpha_n\}$ independent of $u$ and $v$
such that
\begin{equation}
f(u, v)=\sum_{n=0}^\infty
\alpha_n h_n(u, v|q).
\label{careqn6}
\end{equation}
Taking $v=0$ in the above equation and noting the definition of $f(u, v),$ we find that
\[
\frac{1}{(aut, but; q)_\infty}
\sum_{j=0}^k {k \brack j}_q b^j a^{k-j}(aut; q)_j
=\sum_{n=0}^\infty \alpha_n u^n.
\]
Comparing this equation with (\ref{careqn2}), we find that $\alpha_n=t^n h_{n+k}(a, b|q)/(q; q)_n.$
Thus we complete the proof of the theorem.
\end{proof}
%%%%%%%%%%%%%%%%%%%%%%%%%%%%%%%%%%%%%%%%%%%%%%%%%%%%%%%%%%%%%%%%%%%%%%%%%%%%%%%%%%%%%%%%%%%%%%%%%%%%
\section{An extension of Rogers's summation }
%%%%%%%%%%%%%%%%%%%%%%%%%%%%%%%%%%%%%%%%%%%%%%%%%%%%%%%%%%%%%%%%%%%%%%%%%%%%%%%%%%%%%%%%%%%%%%%%%%%%
The Rogers summation formula \cite[p. 44]{Gas+Rah} is one of the most important results for $q$-series,
which can be stated in the following proposition.
\begin{prop} \label{rogerspp}
For $|\alpha abc/q^2|<1,$ we have the summation
\begin{align*}
&{_6 \phi_5} \left({{\alpha, q\sqrt{\alpha}, -q\sqrt{\alpha}, q/a, q/b, q/c}
\atop{\sqrt{\alpha}, -\sqrt{\alpha},\alpha a, \alpha b, \alpha c}}; q, \frac{\alpha abc}{q^2}\right)\\
 &=\frac{(\alpha q, \alpha ab/q, \alpha ac/q, \alpha bc/q; q)_\infty}
{(\alpha a, \alpha b, \alpha c, \alpha abc/q^2; q)_\infty}. \nonumber
\end{align*}
\end{prop}
In \cite{Liu2011}, we give the following extension of the Rogers $_6\phi_5$ summation formula by
using the operator method.
\begin{thm}\label{rogersliuthm}
For $\max\{|\alpha \beta abc/q^2|, |\alpha \gamma abc/q^2|\}<1$,  we have the identity
\begin{align*}
&\sum_{n=0}^\infty \frac{(1-\alpha q^{2n})(\alpha, q/a, q/b, q/c; q)_n}{(q, \alpha a, \alpha b, \alpha c; q)_n}
\(\frac{\alpha abc}{q^2}\)^n  {_4 \phi_3} \left({{q^{-n}, \alpha q^n, \beta, \gamma}
\atop{q/a, q/b,\alpha \beta \gamma ab/q}}; q, q\right)\\
&=\frac{(\alpha, \alpha ac/q, \alpha bc/q, \alpha \beta ab/q, \alpha \gamma ab/q, \alpha \beta \gamma abc/q^2; q)_\infty}
{(\alpha a, \alpha b, \alpha c, \alpha \beta abc/q^2, \alpha \gamma abc/q^2, \alpha \beta \gamma ab/q; q )_\infty}.
\end{align*}
\end{thm}
Now we will use Proposition~\ref{qppmotivation} to give a simple proof of the above theorem.
\begin{proof} Using Watson's $q$-analogue of Whipple's theorem, the Rogers $_6\phi_5$ summation formula and
Tannery's theorem, we can obtain the  asymptotic formula (see \cite[p. 1406 ]{Liu2011} for the details)
for $n\to \infty$,
\[
{_4 \phi_3} \left({{q^{-n}, \alpha q^n, \beta, \gamma}
\atop{q/a, q/b,\alpha \beta \gamma abc/q}}; q, q\right) \sim
\frac{(\gamma, q/a\beta, q/b\beta, ab\alpha \gamma/q; q)_\infty \beta^n}
{(q/a, q/b, \gamma/\beta, ab\alpha \beta \gamma/q; q)_\infty}.
\]
Thus,  by the ratio test,  we know that the left-hand side of the equation in Theorem~\ref{rogersliuthm} converges to an analytic
function of $\beta$ for $|\alpha \beta abc/q^2|<1$.  It is obvious that the left-hand side of the equation in Theorem~\ref{rogersliuthm}
is symmetric in $\beta$ and $\gamma,$ so the left-hand side of the equation in Theorem~\ref{rogersliuthm} also converges to an analytic
function of $\gamma$ for $|\alpha \gamma abc/q^2|<1$. It follows that the left-hand side of the equation in Theorem~\ref{rogersliuthm}
is analytic near $(\beta, \gamma)=(0, 0).$

For simplicity, we temporarily introduce $A_n$ and $f_k(\alpha, \beta)$ as follows
\begin{align*}
&A_n:=\frac{(1-\alpha q^{2n})(\alpha, q/a, q/b, q/c; q)_n}{(q, \alpha a, \alpha b, \alpha c; q)_n}
\(\frac{\alpha abc}{q^2}\)^n, \\
&f_k(\beta, \gamma):=\frac{(\alpha \beta \gamma ab q^k/q; q)_\infty}{(\beta q^k, \gamma q^k, \alpha \beta ab/q, \alpha \gamma ab/q; q)_\infty}.
\end{align*}
Using $f_0(\beta, \gamma)$ to multiply the left-hand side of the equation in Theorem~\ref{rogersliuthm}, we find that
\begin{align*}
f_0(\beta, \gamma) \sum_{n=0}^\infty  {_4 \phi_3} \left({{q^{-n}, \alpha q^n, \beta, \gamma}
\atop{q/a, q/b,\alpha \beta \gamma abc/q}}; q, q\right)A_n\\
=\sum_{n=0}^\infty A_n \sum_{k=0}^n \frac{(q^{-n}, \alpha q^n; q)_k q^k}{(q, q/a, q/b; q)_k}f_k(\beta, \gamma).
\end{align*}
It is obvious that $f_0(\beta, \gamma)$ is analytic near $(\beta, \gamma)=(0, 0)$. Thus, the right-hand side of the
above equation is also analytic near $(\beta, \gamma)=(0, 0)$. By a direct computation, we easily find  that
\[
\partial_{q, \beta}\{f_k(\beta, \gamma)\}=\partial_{q, \gamma}\{f_k(\beta, \gamma)\}
=f_k(\beta, \gamma)\left(\frac{\alpha ab+q^{k+1}-(\beta+\gamma)\alpha abq^k}{q-\alpha \beta \gamma abq^k}\right).
\]
Thus,  there exists a sequence $\{B_n\}$ independent of $\beta$ and $\gamma$ such that
\begin{equation}
f_0(\beta, \gamma) \sum_{n=0}^\infty  {_4 \phi_3} \left({{q^{-n}, \alpha q^n, \beta, \gamma}
\atop{q/a, q/b,\alpha \beta \gamma abc/q}}; q, q\right)A_n
=\sum_{n=0}^\infty B_n h_n(\beta, \gamma|q).\label{rleqn1}
\end{equation}
Setting $\gamma=0$ in the above equation and using $h_n(\beta, 0|q)=\beta^n$,  we obtain
\[
\frac{1}{(\beta, \alpha \beta ab/q; q)_\infty}\sum_{n=0}^\infty  {_3 \phi_2} \left({{q^{-n}, \alpha q^n, \beta}
\atop{q/a, q/b}}; q, q\right)A_n
=\sum_{n=0}^\infty B_n \beta^n.
\]
Ismail, Rahman and Suslov \cite[p. 559, Theorem~5.1]{IsmailRS}  have proved that
\begin{align*}
\sum_{n=0}^\infty  {_3 \phi_2} \left({{q^{-n}, \alpha q^n, \beta}
\atop{q/a, q/b}}; q, q\right)A_n
=\frac{(\alpha, \alpha ac/q, \alpha bc/q, \alpha \beta ab/q; q)_\infty}
{(\alpha a, \alpha b, \alpha c, \alpha \beta abc/q^2; q)_\infty}.
\end{align*}
Comparing the above two equations, we are led to the following identity:
\[
\sum_{n=0}^\infty B_n \beta^n=\frac{(\alpha, \alpha ac/q, \alpha bc/q; q)_\infty}
{(\alpha a, \alpha b, \alpha c, \beta, \alpha \beta abc/q^2; q)_\infty}.
\]
Using the generating function for $h_n$ in Theorem~\ref{gefunthm}, we immediately obtain
\[
\frac{1}{(\beta, \alpha \beta abc/q^2; q)_\infty}=\sum_{n=0}^\infty \frac{\beta^n}{(q; q)_n}
h_n(1, \alpha abc/q^2|q).
\]
Hence we have
\[
B_n=\frac{(\alpha, \alpha ac/q, \alpha bc/q; q)_\infty}
{(\alpha a, \alpha b, \alpha c; q)_\infty}\frac{h_n(1, \alpha abc/q^2|q)}{(q; q)_n}.
\]
Substituting the above equation into (\ref{rleqn1}), we find that the right-hand side of
(\ref{rleqn1}) becomes
\[
\frac{(\alpha, \alpha ac/q, \alpha bc/q; q)_\infty}
{(\alpha a, \alpha b, \alpha c; q)_\infty} \sum_{n=0}^\infty
\frac{1}{(q; q)_n}h_n(\beta, \gamma|q)h_n(1, \alpha abc/q^2|q).
\]
Using the $q$-Mehler formula for $h_n$ in Theorem~\ref{aqmehler}, we easily find that
\[
\sum_{n=0}^\infty
\frac{1}{(q; q)_n}h_n(\beta, \gamma|q)h_n(1, \alpha abc/q^2|q)=\frac{(\alpha\beta\gamma abc/q^2; q)_\infty}
{(\beta, \gamma, \alpha\beta abc/q^2, \alpha\gamma abc/q^2; q)_\infty}.
\]
Combining the above two equations, we find that the right-hand side of (\ref{rleqn1}) equals
\[
\frac{(\alpha, \alpha ac/q, \alpha bc/q, \alpha\beta\gamma abc/q^2; q)_\infty}
{(\alpha a, \alpha b, \alpha c, \beta, \gamma, \alpha\beta abc/q^2, \alpha\gamma abc/q^2; q)_\infty}.
\]
It follows that
\begin{align*}
f_0(\beta, \gamma) \sum_{n=0}^\infty  {_4 \phi_3} \left({{q^{-n}, \alpha q^n, \beta, \gamma}
\atop{q/a, q/b,\alpha \beta \gamma abc/q}}; q, q\right)A_n\\
=\frac{(\alpha, \alpha ac/q, \alpha bc/q, \alpha\beta\gamma abc/q^2; q)_\infty}
{(\alpha a, \alpha b, \alpha c, \beta, \gamma, \alpha\beta abc/q^2, \alpha\gamma abc/q^2; q)_\infty},
\end{align*}
which is equivalent to the equation in Theorem~\ref{rogersliuthm}. This completes the
proof of Theorem~\ref{rogersliuthm}.
\end{proof}
%%%%%%%%%%%%%%%%%%%%%%%%%%%%%%%%%%%%%%%%%%%%%%%%%%%%%%%%%%%%%%%%%%%%%%%%%%%%%%%%%%%%%%%%%%%%%%%%%%%%
\section{An extension of the Andrews-Askey integral }
%%%%%%%%%%%%%%%%%%%%%%%%%%%%%%%%%%%%%%%%%%%%%%%%%%%%%%%%%%%%%%%%%%%%%%%%%%%%%%%%%%%%%%%%%%%%%%%%%%%%
It is known that the Jackson $q$-integral can be defined as
\[
\int_{a}^b f(x)d_q x=(1-q)\sum_{n=0}^\infty [bf(bq^n)-af(aq^n)]q^n.
\]
Using Ramanujan's ${_1}\psi{_1}$ summation, Andrews and Askey \cite{Andrews+Askey}
proved the following $q$-integral  formula.
\begin{prop}\label{AAintegralpp}
If there are no zero factors in the denominator of the integral, then, we have
\[
\int_{u}^v \frac{(qx/u, qx/v; q)_\infty}{(ax, bx; q)_\infty}d_q x
=\frac{(1-q)v(q, u/v, qv/u, abuv; q)_\infty}{(au, bu, av, bv; q)_\infty}.
\]
\end{prop}
To extend the above $q$-integral,  we introduce $H_k(a, b, u, v)$ defined  by
\begin{equation}
H_k(a, b, u, v)=\sum_{r=0}^k {k \brack r}_q\frac{(au, av; q)_r}{(abuv; q)_r}b^ra^{k-r}.
\label{aaeqn1}
\end{equation}
We extend the Andrews and Askey integral to the following more general integral.
\begin{thm}\label{aaliuthm} If there are no zero factors in the denominator of the integral
and $\max\{|au|, |bu|, |cu|, |du|, |av|, |bv|, |cv|, |dv|\}<1,$ then,  we have
\begin{align*}
&\int_{u}^v \frac{(qx/u, qx/v; q)_\infty}{(ax, bx, cx, dx; q)_\infty}d_q x
=\frac{(1-q)v(q, u/v, qv/u, abuv, cduv; q)_\infty}{(au, bu, cu, du,  av, bv, cv, dv; q)_\infty}\\
&\qquad \qquad \qquad \qquad \times \sum_{k=0}^\infty \frac{q^{k(k-1)/2}(-uv)^k}{(q; q)_k}H_k(a, b, u, v)H_k(c, d, u, v).
\end{align*}
\end{thm}
\begin{proof} We temporarily use $I(a, b, c,  d)$ to denote the $q$-integral in the left-hand side of the above equation.
Writing $I(a, b, c,  d)$ in the series form, we easily find that it is an analytic
function of $a, b, c, d$  for
\[
\max\{|au|, |bu|, |cu|, |du|, |av|, |bv|, |cv|, |dv|\}<1.
\]
It is easy to check that $I(a, b, c,  d)$ satisfies the two partial differential equations
\begin{align*}
\partial_{q, a}\{I\}=\partial_{q, b}\{I\}=\int_{u}^v \frac{x(qx/u, qx/v; q)_\infty}{(ax, bx, cx, dx; q)_\infty}d_q x,\\
\partial_{q, c}\{I\}=\partial_{q, d}\{I\}=\int_{u}^v \frac{x(qx/u, qx/v; q)_\infty}{(ax, bx, cx, dx; q)_\infty}d_q x.
\end{align*}
Thus, by (i) in Theorem~\ref{mainthmliu},  there exists a sequence $\{\alpha_{m, n}\}$ such that
\begin{equation}
\int_{u}^v \frac{(qx/u, qx/v; q)_\infty}{(ax, bx, cx, dx; q)_\infty}d_q x
=\sum_{m, n=0}^\infty \alpha_{m, n} h_m(a, b|q)h_n(c, d|q).
\label{aaeqn2}
\end{equation}
Setting $b=d=0$ in the above equation and noting that $h_m(a, 0|q)=a^m$ and $h_n(c, 0|q)=c^n,$
we obtain
\[
\int_{u}^v \frac{(qx/u, qx/v; q)_\infty}{(ax, cx; q)_\infty}d_q x
=\sum_{m, n=-\infty}^\infty \alpha_{m, n} a^mc^n.
\]
Using the Andrews-Askey integral,  we find that the above equation can be written as
\[
\sum_{m, n=0}^\infty \alpha_{m, n} a^mc^n=\frac{(1-q)v(q, u/v, qv/u, acuv; q)_\infty}{(au, cu, av, cv; q)_\infty}.
\]
Using the generating function for $h_n$ in Theorem~\ref{gefunthm}, we find that
the above equation can be written as
\begin{align*}
&\sum_{m, n=0}^\infty \alpha_{m, n} a^mc^n\\
&=(1-q)v(q, u/v, qv/u, acuv; q)_\infty
\sum_{m, n=0}^\infty \frac{h_m(u, v|q)h_n(u, v|q)a^m c^n}
{(q; q)_m(q; q)_n}.
\end{align*}
Equating the coefficients of $a^mc^n$ in the both sides of the above equation, we deduce that
\begin{align*}
\alpha_{m, n}&=(1-q)v(q, u/v, qv/u; q)_\infty\\
&\times \sum_{k \ge 0}\frac{q^{k(k-1)/2}(-uv)^k h_{m-k}(u, v|q)h_{n-k}(u, v|q)}{(q; q)_k (q; q)_{m-k} (q; q)_{n-k}}.
\end{align*}
Substituting the above  equation into (\ref{aaeqn2}) and simplifying, we conclude that
\begin{align}
&\int_{u}^v \frac{(qx/u, qx/v; q)_\infty}{(ax, bx, cx, dx; q)_\infty}d_q x
\label{aaeqn3}\\
&=(1-q)v(q, u/v, qv/u; q)_\infty
\sum_{k=0}^\infty \frac{q^{k(k-1)/2}(-uv)^k}{(q; q)_k}A_kB_k, \nonumber
\end{align}
where $A_k$ and $B_k$ are given by
\[
A_k=\sum_{m=0}^\infty \frac{h_{m+k}(a, b|q)h_m(u, v|q)}{(q; q)_m}  \ \text{and} \
B_k=\sum_{n=0}^\infty \frac{h_{n+k}(c, d|q)h_n(u, v|q)}{(q; q)_n}.
\]
Thus using the Carlitz formula in Theorem~\ref{Carlitzthm},  we immediately find that
\begin{align*}
A_k=\frac{(abuv; q)_\infty}{(au, bu, av, bv; q)_\infty}H_k(a, b, u, v) \\
B_k=\frac{(cduv; q)_\infty}{(cu, du, cv, dv; q)_\infty}H_k(c, d, u, v).
\end{align*}
Substituting the above two equations into (\ref{aaeqn3}), we complete the proof of the theorem.
\end{proof}
The following $q$-integral formula is equivalent to Sears's identity for
the sum of two nonterminating balanced ${}_3\phi_2$ series, which was first noticed
by Al-Salam and Verma \cite{SalamVerma}
\begin{thm}\label{ssvthm} If there are no zero factors in the denominator of the integral
and $\max\{|au|, |av|, |bu|, |bv|, |cu|, |cv|\}<1$, then, we have
 \[
 \int_{u}^v \frac{(qx/u, qx/v, abcuvx; q)_\infty}{(ax, bx, cx; q)_\infty}d_q x
 =\frac{(1-q)v(q, u/v, qv/u, abuv, acuv, bcuv; q)_\infty}{(au, bu, cu, av, bv, cv; q)_\infty}.
 \]
\end{thm}
\begin{proof} Denote $f(a, c)$ as
\[
f(a, c)=\frac{1}{(abuv, bcuv; q)_\infty}\int_{u}^v \frac{(qx/u, qx/v, abcuvx; q)_\infty}{(ax, bx, cx; q)_\infty}d_q x.
\]
By a direct computation, we find that $\partial_{q, a}\{f(a, c)\}=\partial_{q, c}\{f(a, c)\}$ equals
\[
\frac{1}{(abuv, bcuv; q)_\infty}\int_{u}^v (x+buv-(a+c)buvx)\frac{(qx/u, qx/v, qabcuvx; q)_\infty}{(ax, bx, cx; q)_\infty}d_q x.
\]
Hence, by (i) in Proposition~\ref{qppmotivation}, there exists a sequence $\{\alpha_n\}$ independent of $a$ and $c$ such that
\begin{equation}
\frac{1}{(abuv, bcuv; q)_\infty}\int_{u}^v \frac{(qx/u, qx/v, abcuvx; q)_\infty}{(ax, bx, cx; q)_\infty}d_q x
=\sum_{n=0}^\infty \alpha_n h_n(a, c|q).
\label{aaeqn4}
\end{equation}
Putting $c=0$ in the above equation, using $h_n(a, 0|q)=a^n$ and the Andrews-Askey integral,  we find that
 \[
 \sum_{n=0}^\infty \alpha_n a^n=
 =\frac{(1-q)v(q, u/v, qv/u; q)_\infty }{(au, av, bu, bv; q)_\infty}.
 \]
 It follows that
 \[
 \alpha_n =\frac{(1-q)v(q, u/v, qv/u; q)_\infty h_n(u, v|q)}{(q; q)_n( bu, bv; q)_\infty}.
 \]
 Substituting the above equation into (\ref{aaeqn4}) and then using the $q$-Mehler formula for $h_n$, we find that
 \begin{align*}
 &\frac{1}{(abuv, bcuv; q)_\infty}\int_{u}^v \frac{(qx/u, qx/v, abcuvx; q)_\infty}{(ax, bx, cx; q)_\infty}d_q x\\
 &=\frac{(1-q)v(q, u/v, qv/u; q)_\infty }{( bu, bv; q)_\infty}\sum_{n=0}^\infty \frac{h_n(u, v|q) h_n(a, c|q)}{(q; q)_n}\\
 &=\frac{(1-q)v(q, u/v, qv/u, acuv; q)_\infty }{( au, av, bu, bv, cu, cv; q)_\infty}.
 \end{align*}
 Multiplying both sides of the above equation by $(abuv, bcuv; q)_\infty,$ we complete the proof of
 Theorem~\ref{ssvthm}.
\end{proof}
%%%%%%%%%%%%%%%%%%%%%%%%%%%%%%%%%%%%%%%%%%%%%%%%%%%%%%%%%%%%%%%%%%%%%%%%%%%%%%%%%%%%%%%%%%%%%%%%%%%%
\section{Generalizations of Ramanujan's  reciprocity formula }
%%%%%%%%%%%%%%%%%%%%%%%%%%%%%%%%%%%%%%%%%%%%%%%%%%%%%%%%%%%%%%%%%%%%%%%%%%%%%%%%%%%%%%%%%%%%%%%%%%%%
Ramanujan's reciprocity formula  (see, for example, \cite{BCYY}) can be stated as
in the following proposition.
\begin{prop} \label{ramppreciprocity} {\rm (Ramanujan's  reciprocity theorem )}.
\begin{align*}
&v\sum_{n=0}^\infty (-1)^n \frac{q^{n(n+1)/2}(u/v)^n}{(cv; q)_n}
-u\sum_{n=0}^\infty (-1)^n \frac{q^{n(n+1)/2}(v/u)^n}{(cu; q)_n}\label{rreqn1}\\
&\quad=\frac{(v-u)(q, qv/u, qu/v; q)_\infty}{(cu, cv; q)_\infty}.\nonumber
\end{align*}
\end{prop}
By using the $q$-exponential operator to Ramanujan's $_1\psi_1$ summation, we \cite[Theorem~6]{Liu2003} proved
the following reciprocity formula, which is equivalent to an identity of Andrews.
This formula was used to give a simple evaluation of the Askey-Wilson integral \cite{Liu2009}.
\begin{prop}\label{liuppreciprocity} {\rm(Liu \cite[Theorem~6]{Liu2003})}. For $\max\{|bu|, |bv|\}<1,$ we have the reciprocity formula
\begin{align*}
&v\sum_{n=0}^\infty \frac{(q/bu, cduv; q)_n(bv)^n}{(cv, dv; q)_{n+1}}
-u\sum_{n=0}^\infty \frac{(q/bv, cduv; q)_{n}(bu)^n}{(cu, du; q)_{n+1}}\\
&=\frac{(v-u)(q, qv/u, qu/v, bcuv, bduv, cduv; q)_\infty}{(bu, bv, cu, cv, du, dv; q)_\infty}.
\end{align*}
\end{prop}
Ramanujan's reciprocity formula  is the special case $b=d=0$ of Proposition~\ref{liuppreciprocity}.

Using a limiting case of Watson's $q$-analog of Whipple's theorem,
Kang \cite{Kang} found the following equivalent form of Proposition~\ref{liuppreciprocity}.
\begin{prop} \label{kangppreciprocity} {\rm (Kang \cite[Theorem~1.2]{Kang})}. We have
\begin{align*}
&v\sum_{n=0}^\infty (1-q^{2n+1}v/u)\frac{(q/bu, q/cu, q/du; q)_n}{(bv, cv, dv; q)_{n+1}} q^{n(n-1)/2} (-bcduv^2)^n\\
&-u\sum_{n=0}^\infty (1-q^{2n+1}u/v)\frac{(q/bv, q/cv, q/dv; q)_n}{(bu, cu, du; q)_{n+1}} q^{n(n-1)/2} (-bcdvu^2)^n\\
&=(v-u)\frac{(q, qv/u, qu/v, bcuv, bduv, cduv; q)_\infty}{(bu, bv, cu, cv, du, dv; q)_\infty}.
\end{align*}
\end{prop}
For completeness, we will reproduce Kang's proof  here.
\begin{proof}
Watson's $q$-analog of Whipple's
theorem \cite[Eq. (2.5.1)]{Gas+Rah} can be stated as follows
\begin{equation}
\frac{(\alpha q, \alpha ab/q; q)_m}{(\alpha a, \alpha b; q)_m}
{_4\phi_3}\left({{q^{-m}, q/a, q/b, \alpha cd/q }\atop{\alpha c, \alpha d, q^2/{\alpha ab q^m}}}; q, q\right)
\label{wweqn}
\end{equation}
\[
={_8\phi_7}\({{q^{-m}, q\sqrt{\alpha}, -q\sqrt{\alpha}, \alpha, q/a, q/b, q/c, q/d}
\atop{\sqrt{\alpha}, -\sqrt{\alpha}, \alpha a, \alpha b, \alpha c, \alpha d, \alpha q^{m+1}}}; q, \frac{\alpha^2 abcdq^m}{q^2}\).
\]
Setting $a=1$ and then letting $m\to \infty $ in the above equation and simplifying, we find for,  $|\alpha b/q^2|<1$, that
\begin{align*}
&(1-\alpha b/q) \sum_{n=0}^\infty \frac{(q/b, \alpha cd/q; q)_n}{(\alpha c, \alpha d; q)_n} (\alpha b/q)^n
\\
&=\sum_{n=0}^\infty \frac{(1-\alpha q^{2n}) ( q/b, q/c, q/d; q)_n}
{( \alpha b, \alpha c, \alpha d; q)_n} (-1)^n q^{n(n-1)/2} \(\frac{\alpha^2 bcd}{q^2}\)^n.
\end{align*}
Replacing  $(\alpha, b, c, d)$  by $(qv/u, bu, cu, du)$ in the above equation, then dividing both
sides of resulting equation by $(1-bv)(1-cv)(1-dv)$, and finally multiplying both sides by $v$,
we find that the first summation in the left-hand side of the equation in Proposition~\ref{liuppreciprocity}
equlas
\[
v\sum_{n=0}^\infty (1-q^{2n+1}v/u)\frac{(q/bu, q/cu, q/du; q)_n}{(bv, cv, dv; q)_{n+1}} q^{n(n-1)/2} (-bcduv^2)^n.
\]
Interchanging $u$ and $v$, it is found that the second summation in the left-hand side of
the equation in Proposition~\ref{liuppreciprocity} is equal to
\[
u\sum_{n=0}^\infty (1-q^{2n+1}u/v)\frac{(q/bv, q/cv, q/dv; q)_n}{(bu, cu, du; q)_{n+1}} q^{n(n-1)/2} (-bcdu^2v)^n.
\]
Thus, we complete the proof of Proposition~\ref{kangppreciprocity}.
\end{proof}
The following general reciprocity formula was derived by Chu and Zhang from Bailey's $_6\psi_6$ summation,
which is in fact a variant form of Bailey's $_6\psi_6$ summation. We will show that this  reciprocity formula
can be derived from Kang's reciprocity formula by using Proposition~\ref{qppmotivation}. Thus, we
give a new proof of  Bailey's $_6\psi_6$ summation.
\begin{prop}\label{Chuzhangpp} {\rm(Chu and Zhang \cite[Theorem~5]{ChuZhang})}. We have
\begin{align*}
&v\sum_{n=0}^\infty (1-q^{2n+1}v/u)\frac{(q/au, q/bu, q/cu, q/du; q)_n}{(av, bv, cv, dv; q)_{n+1}} (abcdu^2v^2/q)^n\\
&-u\sum_{n=0}^\infty (1-q^{2n+1}u/v)\frac{(q/av, q/bv, q/cv, q/dv; q)_n}{(au, bu, cu, du; q)_{n+1}} (abcdu^2v^2/q)^n\\
&=(v-u)\frac{(q, qv/u, qu/v, abuv, acuv, aduv, bcuv, bduv, cduv; q)_\infty}{(au, av, bu, bv, cu, cv, du, dv, abcdu^2v^2/q; q)_\infty}.
\end{align*}
\end{prop}
\begin{proof}
For the sake of brevity, we first introduce the compact notation
$A_n(a, b, u, v), B_n(u, v)$ and $f(a, b)$ as follows
\begin{align*}
&A_n(a, b, u, v)=\frac{(auq^{-n}, buq^{-n}, avq^{n+1}, bvq^{n+1}; q)_\infty}
{(abuv; q)_\infty},\\
&B_n(u, v)=v (1-q^{2n+1}v/u)\frac{(q,  q/cu, q/du; q)_n}{(cv, dv; q)_{n+1}} (cdv^2/q)^n,\\
&f(a, b)=\sum_{n=0}^\infty A_n(a, b, u, v)B_n(u, v)q^{n^2+n}
-\sum_{n=0}^\infty A_n(a, b, v, u)B_n(v, u) q^{n^2+n}.
\end{align*}
Using the ratio test, we can show that $f(a, b)$ is analytic near $(a, b)=(0, 0).$
By a direct computation, we find that $\partial_{q^{-1}, a}\{f(a, b)\}=\partial_{q^{-1}, b}\{f(a, b)\}$ equals
\begin{align*}
&\frac{1}{(q-abuv)}\sum_{n=0}^\infty A_n(a, b, u, v)B_n(u, v)q^{n^2+n}{\(u/q^{n}+vq^{n+1}-(a+b)uv\)}\\
&-\frac{1}{(q-abuv)}\sum_{n=0}^\infty A_n(a, b, v, u)B_n(v, u) q^{n^2+n}{\(v/q^{n}+uq^{n+1}-(a+b)uv\)}.
\end{align*}
Thus, by (ii) in proposition~\ref{qppmotivation}, there exists a sequence $\{\alpha_n\}$ independent of
$a$ and $b$ such that
\[
f(a, b)=\sum_{n=0}^\infty \alpha_n g_n(a, b|q).
\]
Setting $a=0$ in the above equation,  and using the fact that $g_n(0, b|q)=b^n$, we find that
\begin{align*}
\sum_{n=0}^\infty \alpha_n b^n
&=\sum_{n=0}^\infty B_n(u, v) q^{n^2+n} (buq^{-n}, bvq^{n+1}; q)_\infty\\
&\quad-\sum_{n=0}^\infty B_n(v, u) q^{n^2+n} (bvq^{-n}, buq^{n+1}; q)_\infty.
\end{align*}
By a direct computation, we find, for any complex $z$ and any  integer $n$, that
\begin{equation}
z^n (q/z; q)_n=(-1)^n q^{n(n+1)/2} {(zq^{-n}; q)_\infty}/{(z; q)_\infty}.
\label{rrameqn1}
\end{equation}
It follows that
\begin{align*}
\sum_{n=0}^\infty \alpha_n b^n
&=(bu, bv; q)_\infty\sum_{n=0}^\infty B_n(u, v) (-1)^n q^{n(n+1)/2}\frac{(q/bu; q)_n (bu)^n}{(bv; q)_{n+1}}\\
&\quad-(bu, bv; q)_\infty\sum_{n=0}^\infty B_n(v, u) q^{n(n+1)/2}\frac{(q/bv; q)_n (bv)^n}{(bu; q)_{n+1}}.
\end{align*}
By Proposition~\ref{kangppreciprocity}, we find that the right-hand side of the above equation is equal to
\[
\sum_{n=0}^\infty \alpha_n b^n
=(v-u)\frac{(q, qv/u, qu/v, bcuv, bduv, cduv; q)_\infty}{(cu, cv, du, dv; q)_\infty}.
\]
Using the generating function for $g_n$ in Theorem~\ref{gefunthm}, we find that
\[
\sum_{n=0}^\infty (-1)^n q^{n(n-1)/2}g_n(cuv, duv|q) \frac{b^n}{(q; q)_n}
=(bcuv, bduv; q)_\infty.
\]
Equating the coefficients in the above two equations, we deduce that
\[
\alpha_n=(-1)^n(v-u) q^{n(n-1)/2}\frac{g_n(cuv, duv|q)(q, qv/u, qu/v, cduv; q)_\infty} {(q; q)_n(cu, cv, du, dv; q)_\infty}.
\]
Thus we have
\begin{align*}
f(a, b)&=(v-u)\frac{(q, qv/u, qu/v, cduv; q)_\infty}{(cu, cv, du, dv; q)_\infty}\\
&\qquad \times \sum_{n=0}^\infty (-1)^n q^{n(n-1)/2}
\frac{g_n(a, b|q)g_n(cuv, duv|q)}{(q; q)_n}.
\end{align*}
Using the $q$-Mehler formula for $g_n$ in Theorem~\ref{bqmehler}, we conclude that
\[
\sum_{n=0}^\infty (-1)^n q^{n(n-1)/2}
\frac{g_n(a, b|q)g_n(cuv, duv|q)}{(q; q)_n}
=\frac{(acuv, aduv, bcdu, bcdv; q)_\infty}{(abcdu^2v^2/q; q)_\infty}.
\]
Combining the above two equations, we immediately find that
\[
f(a, b)=(v-u)\frac{(q, qv/u, qu/v, acuv, aduv, bcdu, bcdv, cduv; q )_\infty}
{(cu, cv, du, dv, abcdu^2v^2/q ; q)_\infty},
\]
which is the same as Proposition~\ref{Chuzhangpp} after applying (\ref{rrameqn1})
to the left-hand side of the above equation.
\end{proof}
%%%%%%%%%%%%%%%%%%%%%%%%%%%%%%%%%%%%%%%%%%%%%%%%%%%%%%%%%%%%%%%%%%%%%%%%%%%%%%%%%%%%%%%%%%%%%%%%%%%%
\section{ A general $q$-transformation formula }
%%%%%%%%%%%%%%%%%%%%%%%%%%%%%%%%%%%%%%%%%%%%%%%%%%%%%%%%%%%%%%%%%%%%%%%%%%%%%%%%%%%%%%%%%%%%%%%%%%%%
We proved the following general expansion formula for $q$-series \cite[Theorem~1.1]{Liu2013}.
\begin{thm}\label{newliuthm} If $f(x)$ is an analytic function near $x=0,$
then,  under suitable convergence conditions,  we have
\begin{align*}
&\frac{(\alpha q, \alpha ab/q; q)_\infty}
{(\alpha a, \alpha b; q)_\infty} f(\alpha a)\\
=\sum_{n=0}^\infty & \frac{(1-\alpha q^{2n}) (\alpha, q/a; q)_n (a/q)^n}
{(1-\alpha
)(q, \alpha a; q)_n}
\sum_{k=0}^n \frac{(q^{-n}, \alpha q^n; q)_k q^k}
{(q, \alpha b; q)_k}f(\alpha q^{k+1}).\nonumber
\end{align*}
\end{thm}
The main aim of this section is using the above theorem and Proposition~\ref{qppmotivation} to prove
the following transformation formula for terminating $q$-series.
\begin{thm}\label{newliuthma} If $m$ is an nonnegative integer and  $\{A_n\}$ is a arbitrary complex sequence,
then, we have
\begin{align*}
&\frac{(\alpha q, \alpha ab/q; q)_m}{(\alpha a, \alpha b; q)_m}
\sum_{n=0}^m \frac{(q^{-m}, q/a, q/b; q)_nq^n}{(q^2/{\alpha ab q^m}; q)_n}A_n\\
&=\sum_{n=0}^m \frac{(1-\alpha q^{2n})(q^{-m}, \alpha, q/a, q/b; q)_n(\alpha ab q^{m-1})^n}
{(1-\alpha)(q, \alpha q^{m+1}, \alpha a, \alpha b; q)_n}
\sum_{k=0}^n (q^{-n}, \alpha q^n; q)_k q^k A_k.
\end{align*}
\end{thm}
To prove Theorem~\ref{newliuthma}, we first prove the following lemma by using Theorem~\ref{newliuthm}.
\begin{lem} \label{liulem} For any nonnegative integer $m$ and arbitrary complex sequence $\{A_n\}$, we have
 \begin{align*}
 &\frac{(\alpha q, \alpha ab/q; q)_\infty}
{(\alpha a, \alpha b; q)_\infty} \sum_{n=0}^m A_n (q^{-m}, q/a; q)_n (\alpha a q^m)^n\\
&=\sum_{n=0}^\infty \frac{(1-\alpha q^{2n})(\alpha, q/a, q/b; q)_n (-\alpha ab/q)^n q^{n(n-1)/2}}
{(1-\alpha) (q, \alpha a, \alpha b; q)_n}\\
&\quad \times \sum_{l=0}^n \frac{(q^{-m}, q^{-n}, \alpha q^n; q)_l}{(q/b; q)_l}
\left(\frac{q^{m+2}}{b}\right)^l A_l.
 \end{align*}
 \end{lem}
\begin{proof}
Suppose that $m$ is an nonnegative integer and $\{A_n\}$ is a arbitrary complex sequence.
In Theorem \ref{newliuthm},  we can choose $f(x)$ as follows
\[
f(x)=\sum_{l=0}^m  A_l (q^{-m}, q\alpha/x; q )_l (q^m x)^l.
\]
Letting $x=\alpha a$ in the above equation, we immediately find that
\begin{equation}
f(\alpha a)=\sum_{l=0}^m  A_l (q^{-m}, q/a; q )_l (\alpha a q^m)^l.
\label{teqn1}
\end{equation}
It is easy to check that $(q^{-k}; q)_l=0$ for $l>k.$ Thus we at once deduce that
\begin{align*}
f(\alpha q^{k+1})=\sum_{l=0}^m A_l (q^{-m}, q^{-k}; q)_l (\alpha q^{m+k+1})^l
=\sum_{l=0}^k A_l (q^{-m}, q^{-k}; q)_l (\alpha q^{m+k+1})^l.
\end{align*}
It follows that
\[
\sum_{k=0}^n \frac{(q^{-n}, \alpha q^n; q)_k}{(q, \alpha b; q)_k} f(\alpha q^{k+1})
=\sum_{k=0}^n \frac{(q^{-n}, \alpha q^n; q)_k}{(q, \alpha b; q)_k}\sum_{l=0}^k A_l (q^{-m}, q^{-k}; q)_l (\alpha q^{m+k+1})^l.
\]
Interchanging the order of the summation on the right-hand side of the above equation,
 we find that the right-hand side of the above equation becomes
 \[
 \sum_{l=0}^n A_l (q^{-m}; q)_l (\alpha q^{m+1})^l \sum_{k=l}^n \frac{(q^{-n}, \alpha q^n; q)_k (q^{-k}; q)_l}
 {(q, \alpha b; q)_k} q^{k(l+1)}.
 \]
Using $(q^{-k}; q)_l (q; q)_{k-l}=(-1)^l (q; q)_k q^{l(l-1)/2-kl}$, we
find that the above equation becomes
\[
 \sum_{l=0}^n A_l (q^{-m}; q)_l (-\alpha q^{m+1})^l q^{l(l-1)/2}\sum_{k=l}^n \frac{(q^{-n}, \alpha q^n; q)_k q^k }
 {(\alpha b; q)_k (q; q)_{k-l}}.
 \]
 Making the variable change  $k-l=j$ in the above equation,  we deduce  that
 \begin{equation}
 \sum_{l=0}^n A_l \frac{(q^{-n}, q^{-m}; q)_l}{(\alpha b; q)_l} (-\alpha q^{m+1})^l q^{l(l+1)/2}
 \sum_{j=0}^{n-l} \frac{(q^{-n+l}, \alpha q^{n+l}; q)_j q^j }
 {(\alpha b q^l; q)_j (q; q)_{j}}.
 \label{teqn2}
 \end{equation}
 Using the $q$-Chu-Vandermonde summation formula, we find that the inner summation of the above equation equals
 \[
 \frac{(bq^{-n}; q)_{n-l} (\alpha q^{n+l})^{n-l}}{(\alpha b q^l; q)_{n-l}}
 =(-\alpha b)^{n-l} q^{n(n-1)/2-l(l-1)/2} \frac{(q/b; q)_n (\alpha b; q)_l}{(q/b; q)_l (\alpha b; q)_n}.
 \]
 Substituting this equation into (\ref{teqn2}), we conclude that
 \begin{align}
 &\sum_{k=0}^n \frac{(q^{-n}, \alpha q^n; q)_k}{(q, \alpha b; q)_k} f(\alpha q^{k+1})
 \label{teqn3}\\
&= \frac{(q/b; q)_n}{(\alpha b; q)_n} (-\alpha b; q)^n q^{n(n-1)/2}
 \sum_{l=0}^n \frac{(q^{-n}, q^{-m}, \alpha b q^l; q)_l}{(q/b; q)_l}
 \left(\frac{q^{m+2}}{b}\right)^l A_l.\nonumber
 \end{align}
 Substituting the above equation and  equation (\ref{teqn1}) into Theorem \ref{newliuthm}, we complete
 the proof of Lemma~\ref{liulem}.
\end{proof}
Now we begin to prove Theorem~\ref{newliuthma} by using the above lemma and Proposition~\ref{qppmotivation}.
\begin{proof}
Setting $b=q^{m+1}$ in Lemma~\ref{liulem} and then replacing $A_n$ by $A_n (q/b; q)_n$, we immediately deduce that
\begin{align}
&\frac{(\alpha q; q)_m}{(\alpha a; q)_m} \sum_{n=0}^{m} A_n (q^{-m},  q/a, q/b; q)_n (\alpha a q^m)^n
\label{teqn4}\\
=\sum_{n=0}^m & \frac{(1-\alpha q^{2n}) (q^{-m}, \alpha, q/a; q)_n (-\alpha aq^{m})^n q^{n(n-1)/2}}
{(1-\alpha)(q, \alpha a, \alpha q^{m+1}; q)_n}B_n,\nonumber
\end{align}
where $B_n$ is defined as
\[
B_n=\sum_{k=0}^n (q^{-n}, \alpha q^n; q)_k q^k A_k.
\]
For simplicity, we temporarily introduce $f_{m,n}(a, b), f(a, b)$ and $g(a, b)$ as follows
\begin{align*}
&f_{m, n} (a, b)=\frac{(aq^{-n}, bq^{-n}, \alpha aq^m, \alpha b q^m; q)_\infty}
{(\alpha ab q^{m-n-1}; q)_\infty},\\
&f(a, b)=\sum_{n=0}^m  \frac{(1-\alpha q^{2n}) (q^{-m}, \alpha; q)_n \alpha^n q^{n^2+mn}}
{(1-\alpha) (q, \alpha q^{m+1}; q)_n}f_{n, n} (a, b)B_n,\\
&g(a, b)=(\alpha q; q)_m \sum_{n=0}^{m} (q^{-m}; q)_n (-\alpha q^m)^n q^{n(n+1)/2}f_{m, n}(a, b)A_n.
\end{align*}
It is easily seen that
\begin{align*}
f(a, 0)&=(a, \alpha a; q)_\infty \sum_{n=0}^m \frac{(1-\alpha q^{2n}) (q^{-m}, \alpha, q/a; q)_n (-\alpha aq^{m})^n q^{n(n-1)/2}}
{(1-\alpha)(q, \alpha a, \alpha q^{m+1}; q)_n}B_n,\\
g(a, 0)&=(a, \alpha a; q)_\infty  \frac{(\alpha q; q)_m}{(\alpha a; q)_m} \sum_{n=0}^{m} A_n (q^{-m},  q/a; q)_n (\alpha a q^m)^n.
\end{align*}
Combining the above two equations and (\ref{teqn4}), we find the identity, $f(a, 0)=g(a, 0).$

By a direct computation, we find that
\[
\partial_{q^{-1}, a} \{f_{m, n}(a, b)\}=\partial_{q^{-1}, b} \{f_{m, n}(a, b)\}=\frac{(1+\alpha q^{m+n}-\alpha(a+b)q^{m-1})f_{m, n} (a, b)}{q^{n+1}(1-\alpha ab q^{m-n-2})}.
\]
It follows that
\[
\partial_{q^{-1}, a} \{f(a, b)\}=\partial_{q^{-1}, b} \{f(a, b)\}, ~\text{and}~
\partial_{q^{-1}, a} \{g(a, b)\}=\partial_{q^{-1}, b} \{g(a, b)\}.
\]
Thus, by (ii) in Proposition~\ref{qppmotivation}, there exist two sequences $\{\beta_l\}$ and $\{\gamma_l\}$
independent of $a$ and $b$ such that
\begin{align*}
f(a, b)=\sum_{l=0}^\infty \beta_l h_l(a, b|q),\quad
g(a, b)=\sum_{l=0}^\infty \gamma_l h_l(a, b|q).
\end{align*}
Setting $b=0$ in the above equations and  using $f(a, 0)=g(a, 0),$ we find that $\beta_l=\gamma_l$ for any nonnegative
integer $l$. It follows that $f(a, b)=g(a, b),$ which gives
\begin{align*}
&\frac{(\alpha q; q)_m}{(\alpha a, \alpha b; q)_m}\sum_{n=0}^m (q^{-m}, q/a, q/b; q)_n q^{-n(n+1)/2}(-\alpha ab q^m)^n (\alpha ab/q; q)_{m-n}A_n\\
&=\sum_{n=0}^m \frac{(1-\alpha)(q^{-m}, \alpha, q/a, q/b; q)_n(\alpha ab q^{m-1})^n}{(1-\alpha)(q, \alpha q^{m+1}, \alpha a, \alpha b; q)_n}B_n.
\end{align*}
From the definition of the $q$-shifted factorial and by a direct computation, we find that
\[
(\alpha ab/q; q)_{m-n}= \frac{(-1)^nq^{n(n+1)/2}(\alpha ab/q; q)_m}{(q^2/\alpha ab q^m; q)_n (\alpha ab q^{m-1})^n}.
\]
Combining the above two equations, we finish the proof of Theorem~\ref{newliuthma}.
\end{proof}
%%%%%%%%%%%%%%%%%%%%%%%%%%%%%%%%%%%%%%%%%%%%%%%%%%%%%%%%%%%%%%%%%%%%%%%%%%%%%%%%%%%%%%%%%%%%%%%%%%%%
\section{An extension of  Watson's $q$-analog of Whipple's theorem  }
%%%%%%%%%%%%%%%%%%%%%%%%%%%%%%%%%%%%%%%%%%%%%%%%%%%%%%%%%%%%%%%%%%%%%%%%%%%%%%%%%%%%%%%%%%%%%%%%%%%%
Taking $A_k=(\beta, \gamma; q)_kz^k/(q, c, d, h; q)_k$ in Theorem~\ref{newliuthma},
we immediately obtain the following theorem.
\begin{thm} \label{newliuthmb} For any nonnegative integer $m$, we have the $q$-formula
\begin{align*}
&\frac{(\alpha q, \alpha ab/q; q)_m}{(\alpha a, \alpha b; q)_m}
{_5\phi_4}\left({{q^{-m}, q/a, q/b, \beta, \gamma }\atop{q^2/{\alpha ab q^m, c,  d, h}}}; q, qz\right)\\
=&\sum_{n=0}^m \frac{(1-\alpha q^{2n})(q^{-m}, \alpha, q/a, q/b; q)_n(\alpha ab q^{m-1})^n}
{(1-\alpha)(q, \alpha q^{m+1}, \alpha a, \alpha b; q)_n}
{_4\phi_3}\({{q^{-n}, \alpha q^n, \beta, \gamma}\atop{c, d, h}}; q, qz\).
\end{align*}
\end{thm}
If we take $h=\gamma=0$ and $z=1$ in the above equation, we immediately deduce the
following proposition.
\begin{prop}\label{Exwwpp} For any nonnegative integer $m$, then, we have
\begin{align*}
&\frac{(\alpha q, \alpha ab/q; q)_m}{(\alpha a, \alpha b; q)_m}
{_4\phi_3}\left({{q^{-m}, q/a, q/b, \beta}\atop{q^2/{\alpha ab q^m, c,  d}}}; q, q\right)\\
=&\sum_{n=0}^m \frac{(1-\alpha q^{2n})(q^{-m}, \alpha, q/a, q/b; q)_n(\alpha ab q^{m-1})^n}
{(1-\alpha)(q, \alpha q^{m+1}, \alpha a, \alpha b; q)_n}
{_3\phi_2}\({{q^{-n}, \alpha q^n, \beta}\atop{c, d}}; q, q\).
\end{align*}
\end{prop}
This proposition contains Watson's $q$-analog of Whipple's theorem as a special case. Thus
we may regard it an extension of Watson's $q$-analog of Whipple's theorem.
\begin{prop}\label{wwpp} {\rm (Watson's $q$-analog of Whipple's theorem)}.
\begin{align*}
&\frac{(\alpha q, \alpha ab/q; q)_m}{(\alpha a, \alpha b; q)_m}
{_4\phi_3}\left({{q^{-m}, q/a, q/b, \alpha cd/q }\atop{\alpha c, \alpha d, q^2/{\alpha ab q^m}}}; q, q\right)\\
&={_8\phi_7}\({{q^{-m}, q\sqrt{\alpha}, -q\sqrt{\alpha}, \alpha, q/a, q/b, q/c, q/d}
\atop{\sqrt{\alpha}, -\sqrt{\alpha}, \alpha a, \alpha b, \alpha c, \alpha d, \alpha q^{m+1}}}; q, \frac{\alpha^2 abcdq^m}{q^2}\).
\end{align*}
\end{prop}
\begin{proof}
If $(c, d, \beta)$ is replaced by $(\alpha c, \alpha d, \alpha cd/q)$, we find that
the left-hand side of the equation in Proposition~\ref{Exwwpp} becomes
the left-hand side of the equation in Proposition~\ref{wwpp}, and
the right-hand side becomes
\[
\sum_{n=0}^m \frac{(1-\alpha q^{2n})(q^{-m}, \alpha, q/a, q/b; q)_n(\alpha ab q^{m-1})^n}
{(1-\alpha)(q, \alpha q^{m+1}, \alpha a, \alpha b; q)_n}
{_3\phi_2}\({{q^{-n}, \alpha q^n, \alpha cd/q}\atop{\alpha c, \alpha d}}; q, q\).
\]
Using the $q$-Pfaff-Saalsch\"{u}tz formula (see,  for example \cite[p. 13,  Eq. (1.7.2)]{Gas+Rah}),  we find that
\[
{_3 \phi_2} \left({{q^{-n}, \alpha q^n, \alpha cd/q}\atop{\alpha c, \alpha d}}; q,
q\right)=\frac{(q/c, q/d; q)_n}{(\alpha c, \alpha d; q)_n}\left(\frac{\alpha cd}{q}\right)^n.
\]
Combining the above two equations we arrive at the right-hand side of the equation in Proposition~\ref{wwpp}.
Thus we complete the proof of the proposition.
\end{proof}
Letting $m\to \infty$ in  Theorem \ref{newliuthma}, we immediately obtain the following theorem.
Many important applications of this theorem to mock-theta function identities have been discussed
in the paper~\cite{Liumocktheta}.
\begin{thm}\label{newliuthmc} For $|\alpha abz/q|<1$, we have the $q$-transformation formula
\begin{align*}
&\frac{(\alpha q, \alpha ab/q; q)_\infty}
{(\alpha a, \alpha b; q)_\infty} {_4\phi_3} \left({{q/a, q/b, \beta, \gamma} \atop { c, d, h}} ;  q, \frac{\alpha ab z}{q} \right) \\
&=\sum_{n=0}^\infty \frac{(1-\alpha q^{2n}) (\alpha, q/a, q/b; q)_n (-\alpha ab/q)^n q^{n(n-1)/2}}
{(1-\alpha)(q, \alpha a, \alpha b; q)_n}
{_4\phi_3} \left({{q^{-n}, \alpha q^n, \beta, \gamma} \atop {c, d, h}} ;  q, qz \right).
\end{align*}
\end{thm}
%%%%%%%%%%%%%%%%%%%%%%%%%%%%%%%%%%%%%%%%%%%%%%%%%%%%%%%%%%%%%%%%%%%%%%%%%%%%%%%%%%%%%%%%%%%%%%%%%%%%
\section{Some $q$-series transformation formulas}
%%%%%%%%%%%%%%%%%%%%%%%%%%%%%%%%%%%%%%%%%%%%%%%%%%%%%%%%%%%%%%%%%%%%%%%%%%%%%%%%%%%%%%%%%%%%%%%%%%%%
In this section we will use Theorem~\ref{newliuthmb} to derive some $q$-transformation formula.
\begin{thm}\label{liutfthma} If $m$ is an nonnegative integer, then, we have
\begin{align*}
&\frac{(\alpha^2 q^2, \alpha^2 ab/q^2; q^2)_m}{(\alpha^2 a, \alpha^2 b; q^2)_m}
{_5\phi_4}\left({{q^{-2m}, q^2/a, q^2/b, \lambda, q\lambda }\atop{\alpha,  q\alpha, q^2\lambda^2,  q^4/{\alpha^2 ab q^{2m}}}}; q^2, q^2 \right)\\
=&\sum_{n=0}^m \frac{(1-\alpha^2 q^{4n})(q^{-2m}, \alpha^2, q^2/a, q^2/b; q^2)_n (-q, \alpha /\lambda; q)_n (\alpha^2 \lambda ab q^{2m-2})^n}
{(1-\alpha^2)(q^2, \alpha^2 q^{2m+2}, \alpha^2 a, \alpha^2 b; q^2)_n (\alpha, -q\lambda; q)_n}.
\end{align*}
\end{thm}
\begin{proof}
We first replace $q$ by $q^2$ and $\alpha$ by $\alpha^2$ in Theorem \ref{newliuthmb} and then set $(\beta, \gamma, c, d, h, z)
=(\lambda, q\lambda, \alpha, q\alpha, \lambda^2 q^2, 1)$ in the resulting equation, we deduce that
\begin{align*}
&\frac{(\alpha^2 q^2, \alpha^2 ab/q^2; q^2)_m}{(\alpha^2 a, \alpha^2 b; q^2)_m}
{_5\phi_4}\left({{q^{-2m}, q^2/a, q^2/b, \lambda, q\lambda }\atop{\alpha,  q\alpha, q^2\lambda^2,  q^4/{\alpha^2 ab q^{2m}}}}; q^2, q^2 \right)\\
=&\sum_{n=0}^m \frac{(1-\alpha^2 q^{4n})(q^{-2m}, \alpha^2, q^2/a, q^2/b; q^2)_n(\alpha^2 ab q^{2m-2})^n}
{(1-\alpha^2)(q^2, \alpha^2 q^{2m+2}, \alpha^2 a, \alpha^2 b; q^2)_n}\\
&\qquad\times{_4\phi_3}\({{q^{-2n}, \alpha^2 q^{2n}, \lambda, q\lambda}\atop{\alpha, q\alpha, q^2\lambda^2}}; q^2, q^2 \).
\end{align*}
Verma and Jain \cite[Eq. (5.3)]{VermaJain} (see also \cite[p. 110, Ex. (3.34)]{Gas+Rah}) proved that
\begin{equation}
{_4\phi_3} \left({{q^{-2n}, \alpha^2 q^{2n}, \lambda, q\lambda} \atop { \alpha, q \alpha, q^2\lambda^2}} ;  q^2, q^2 \right)
=\frac{\lambda^n (-q, \alpha /\lambda; q)_n}{(\alpha, -q\lambda; q)_n}.
\label{rogers:eqn1}
\end{equation}
Combining the above two equations, we complete the proof of Theorem~\ref{liutfppa}.
\end{proof}
Letting $m\to \infty$ in Theorem~\ref{liutfthma}, we obtain the following proposition.
\begin{prop}\label{liutfppa}
\begin{align*}
&\frac{(\alpha^2 q^2, \alpha^2 ab/q^2; q^2)_\infty}{(\alpha^2 a, \alpha^2 b; q^2)_\infty}
{_4\phi_3}\left({{q^2/a, q^2/b, \lambda, q\lambda }\atop{\alpha,  q\alpha, q^2\lambda^2 }}; q^2, \frac{\alpha^2 ab}{q^2} \right)\\
=&\sum_{n=0}^\infty \frac{(1-\alpha^2 q^{4n})( \alpha^2, q^2/a, q^2/b; q^2)_n (-q, \alpha /\lambda; q)_n (-\alpha^2 \lambda ab)^n q^{n^2-3n}}
{(1-\alpha^2)(q^2, \alpha^2 a, \alpha^2 b; q^2)_n (\alpha, -q\lambda; q)_n}.
\end{align*}
\end{prop}
Let $(a, b, \alpha, \gamma)=(0,  0, - q,  0)$ in Proposition~\ref{liutfppa}, we have the evaluation
\begin{align}
\sum_{n=0}^\infty \frac{q^{2n^2+2n}}{(-q; q)_{2n}(q^2; q^2)_n}
&=\frac{1}{(q^2; q^2)_\infty} \sum_{n=-\infty}^\infty (-1)^n  q^{\frac{7n^2+3n}{2}}
\label{rrideqn1}\\
=\frac{(q^2, q^5, q^7; q^7)_\infty}{(q^2; q^2)_\infty}.\nonumber
\end{align}

\begin{thm}\label{liutfthmb} For any nonnegative integer $m,$ we have
\begin{align*}
&\frac{(\alpha^2 q^2, \alpha^2 ab/q^2; q^2)_m}{(\alpha^2 a, \alpha^2 b; q^2)_m}
{_5\phi_4}\left({{q^{-2m}, q^2/a, q^2/b, \lambda, q\lambda }\atop{q\alpha,  q^2\alpha, \lambda^2,  q^4/{\alpha^2 ab q^{2m}}}}; q^2, q^2 \right)\\
=&\sum_{n=0}^m \frac{(1+\alpha q^{2n})(q^{-2m}, \alpha^2, q^2/a, q^2/b; q^2)_n (-q, q\alpha /\lambda; q)_n (\alpha^2 \lambda ab q^{2m-2})^n}
{(1+\alpha)(q^2, \alpha^2 q^{2m+2}, \alpha^2 a, \alpha^2 b; q^2)_n (\alpha, -\lambda; q)_n}.
\end{align*}
\end{thm}
\begin{proof}
Replacing $q$ by $q^2$ and $\alpha$ by $\alpha^2$ in Theorem \ref{newliuthmb} and then set $(\beta, \gamma, c, d, h, z)
=(\lambda, q\lambda, q\alpha, q^2\alpha, \lambda^2 , 1)$ in the resulting equation, we deduce that
\begin{align*}
&\frac{(\alpha^2 q^2, \alpha^2 ab/q^2; q^2)_m}{(\alpha^2 a, \alpha^2 b; q^2)_m}
{_5\phi_4}\left({{q^{-2m}, q^2/a, q^2/b, \lambda, q\lambda }\atop{q\alpha,  q^2\alpha, \lambda^2,  q^4/{\alpha^2 ab q^{2m}}}}; q^2, q^2 \right)\\
=&\sum_{n=0}^m \frac{(1-\alpha^2 q^{4n})(q^{-2m}, \alpha^2, q^2/a, q^2/b; q^2)_n(\alpha^2 ab q^{2m-2})^n}
{(1-\alpha^2)(q^2, \alpha^2 q^{2m+2}, \alpha^2 a, \alpha^2 b; q^2)_n}\\
&\qquad\times{_4\phi_3}\({{q^{-2n}, \alpha^2 q^{2n}, \lambda, q\lambda}\atop{q\alpha, q^2\alpha, \lambda^2}}; q^2, q^2 \).
\end{align*}
Verma and Jain \cite[Eq. (5.4)]{VermaJain}  derived  the following series summation  formula:
\begin{equation}
{_4\phi_3} \left({{q^{-2n}, \alpha^2 q^{2n}, \lambda, q\lambda} \atop { q\alpha, q^2 \alpha, \lambda^2}} ;  q^2, q^2 \right)
=\frac{\lambda^n (-q, q\alpha /\lambda; q)_n (1-\alpha)}{(\alpha, -\lambda; q)_n (1-\alpha q^{2n})}.
\label{rogers:eqn2}
\end{equation}
Combining the above two equations, we finish the proof of Theorem~\ref{liutfthmb}.
\end{proof}
Letting $m\to \infty$ in Theorem~\ref{liutfthmb}, we obtain the following proposition.
\begin{prop}\label{liutfppb}
\begin{align*}
&\frac{(\alpha^2 q^2, \alpha^2 ab/q^2; q^2)_\infty}{(\alpha^2 a, \alpha^2 b; q^2)_\infty}
{_4\phi_3}\left({{q^2/a, q^2/b, \lambda, q\lambda }\atop{q\alpha,  q^2\alpha, \lambda^2}}; q^2, \frac{\alpha^2 ab}{q^2} \right)\\
=&\sum_{n=0}^\infty \frac{(1+\alpha q^{2n})(\alpha^2, q^2/a, q^2/b; q^2)_n (-q, q\alpha /\lambda; q)_n (-\alpha^2 \lambda ab )^nq^{n^2-3n}}
{(1+\alpha)(q^2,  \alpha^2 a, \alpha^2 b; q^2)_n (\alpha, -\lambda; q)_n}.
\end{align*}
\end{prop}
Setting $(a, b, \alpha, \gamma)=(0, 0, -q, 0)$ and $(0, 0, -1, 0)$ in Proposition~\ref{liutfppb}, respectively,
we find that
\begin{align}
\sum_{n=0}^\infty \frac{q^{2n^2+2n}}{(-q; q)_{2n+1}(q^2; q^2)_n}
&=\frac{(q, q^6, q^7; q^7)_\infty}{(q^2; q^2)_\infty},\label{rrideqn2}\\
\sum_{n=0}^\infty \frac{q^{2n^2}}{(-q; q)_{2n}(q^2; q^2)_n}
&=\frac{(q^3, q^4, q^7; q^7)_\infty}{(q^2; q^2)_\infty}. \label{rrideqn3}
\end{align}
Identities (\ref{rrideqn1}), (\ref{rrideqn2}) and (\ref{rrideqn3}) are called the Rogers-Selberg identities
(see, for example, \cite[Eqs. (2.7.1), (2.7.2), (2.7. 3)]{McLaughlinS}).

It is easily seen that the identity of Verma and  Jain \cite[Eq. (2.28)]{VermaJain} is equivalent to the summation
\begin{equation}
{_3\phi_2} \left({{q^{-n}, \alpha q^n, 0} \atop {\sqrt{q\alpha}, -\sqrt{q\alpha}}} ;  q, q \right)
=\begin{cases} 0 &\text {if $n$ is odd}\\
(-1)^l q^{l^2} \frac{(q; q^2)_l \alpha^l }{(q\alpha; q^2)_l}, & \text{if $n=2l$}
\end{cases}
\label{rogers:eqn3}
\end{equation}
Setting $(\beta, \gamma, h, c, d, z)=(0, 0, 0, \sqrt{q\alpha}, -\sqrt{q\alpha}, 1)$ in Theorem~\ref {newliuthmb} and
then using (\ref{rogers:eqn3}), we obtain the following theorem.
\begin{thm}\label{liutfthmc} For any nonnegative integer $m,$ we have
\begin{align*}
&\frac{(\alpha q, \alpha ab/q; q)_m}{(\alpha a, \alpha b; q)_m}
{_4\phi_3}\left({{q^{-m}, q/a, q/b, 0 }\atop{\sqrt{q\alpha}, -\sqrt{q\alpha},  q^2/{\alpha ab q^m}}}; q, q\right)\\
&=\sum_{n=0}^{[m/2]} \frac{(1-\alpha q^{4n})(q^{-m}, \alpha, q/a, q/b; q)_{2n} (q; q^2)_{n}(-\alpha^3 a^2b^2 q^{2m-2})^n q^{n^2}}
{(1-\alpha)(q, \alpha q^{m+1}, \alpha a, \alpha b; q)_{2n} (q\alpha; q^2)_n}.
\end{align*}
\end{thm}
Letting $m\to \infty$ in Theorem~\ref{liutfthmc}, we obtain the following proposition.
\begin{prop}\label{liutfppc}
\begin{align*}
&\frac{(\alpha q, \alpha ab/q; q)_\infty}{(\alpha a, \alpha b; q)_\infty}
{_3\phi_2}\left({{q/a, q/b, 0 }\atop{\sqrt{q\alpha}, -\sqrt{q\alpha}}}; q, \frac{\alpha ab}{q}\right)\\
&=\sum_{n=0}^{\infty} \frac{(1-\alpha q^{4n})(\alpha, q/a, q/b; q)_{2n} (q; q^2)_{n}(-\alpha^3 a^2b^2 )^n q^{3n^2-3n}}
{(1-\alpha)(q, \alpha a, \alpha b; q)_{2n} (q\alpha; q^2)_n}.
\end{align*}
\end{prop}
Putting $(a, b, \alpha)=(0, 0, 1)$ and $(0, 0, q^2)$ in Proposition~\ref{liutfppc}, respectively, we obtain
the Rogers identities (see, for example, \cite[Eqs. (2.14.1), (2.14.3)]{McLaughlinS})
\begin{align}
\sum_{n=0}^\infty \frac{q^{n^2}} {(q; q)_n (q; q^2)_n}&=\frac{(q^6, q^8, q^{14}; q^{14})_\infty} {(q; q)_\infty},\label{rrideqn4}\\
\sum_{n=0}^\infty \frac{q^{n^2+2n}}{(q; q)_n (q; q^2)_{n+1}}&=\frac{(q^2, q^{12}, q^{14}; q^{14})_\infty}{(q; q)_\infty}. \label{rrideqn5}
\end{align}
%%%%%%%%%%%%%%%%%%%%%%%%%%%%%%%%%%%%%%%%%%%%%%%%%%%%%%%%%%%%%%%%%%%%%%%%%%%%%%%%%%%%%%%%%%%%%%%%%%%%
\section{A multilinear generating function for the Rogers-Szeg\H{o} polynomials }
%%%%%%%%%%%%%%%%%%%%%%%%%%%%%%%%%%%%%%%%%%%%%%%%%%%%%%%%%%%%%%%%%%%%%%%%%%%%%%%%%%%%%%%%%%%%%%%%%%%%
\begin{thm} \label{mgrogers} If $\max\{ |a|, |c|,  |x_1|, |y_1|, \ldots, |x_k|, |y_k|\}<1,$ then, we have the
following multilinear generating function for the Rogers-Szeg\H{o} polynomials:
\begin{align*}
\sum_{n_1, n_2, \ldots, n_k=0}^\infty \frac{(a; q)_{n_1+n_2+\cdots+n_k}h_{n_1}(x_1, y_1|q) h_{n_2}(x_2, y_2|q)\cdots h_{n_k}(x_k, y_k|q)}
{(c; q)_{n_1+n_2+\cdots+n_k}(q; q)_{n_1}(q; q)_{n_2}\cdots (q; q)_{n_k}}\\
=\frac{(a; q)_\infty}{(c, x_1, y_1, x_2, y_2, \ldots, x_k, y_k)_\infty}
{_{2k+1}\phi_{2k}}\left({{c/a, x_1, y_1, x_2, y_2, \ldots, x_k, y_k}
\atop{0, 0, \ldots, 0}}; q, a\right).
\end{align*}
\end{thm}
\begin{proof}
If we use  $f(x_1, y_1, \ldots, x_k, y_k)$ to denote the right-hand side of the above equation, then,
using the ratio test, we find that $f$ is an analytic function of
$x_1, y_1, x_2, y_2, \ldots, x_k, y_k$ for $\max\{ |a|,|c|,  |x_1|, |y_1|, \ldots, |x_k|, |y_k|\}<1.$
By a direct computation, we deduce that for $j=1, 2\ldots, k$, $\partial_{q, x_j}\{f\}=\partial_{q, y_j}\{f\}$
equals
\[
 \frac{(a; q)_\infty}{(c, x_1, y_1, x_2, y_2, \ldots, x_k, y_k)_\infty}
{_{2k+1}\phi_{2k}}\left({{c/a, x_1, y_1, x_2, y_2, \ldots, x_k, y_k}
\atop{0, 0, \ldots, 0}}; q, qa\right).
\]
Thus, by (i) in Theorem~\ref{mainthmliu},  there exists a sequence $\{\alpha_{n_1, n_2, \ldots, n_k}\}$ independent
of $x_1, y_1, \ldots, x_k, y_k$ such that
\begin{align}
&f(x_1, y_1, x_2, y_2, \ldots, x_k, y_k)\label{mgeneqn}\\
&=\sum_{n_1, n_2, \ldots, n_k=0}^\infty \alpha_{n_1, n_2, \ldots, n_k}
h_{n_1}(x_1, y_1|q) h_{n_2}(x_2, y_2|q)\cdots h_{n_k}(x_k, y_k|q).\nonumber
\end{align}
Setting $y_1=y_2\cdots=y_k=0$ in the above equation, we immediately obtain
\begin{align}
&\sum_{n_1, n_2, \ldots, n_k=0}^\infty \alpha_{n_1, n_2, \ldots, n_k}
x_{1}^{n_1}x_2^{n_2}\cdots x_{k}^{n_k}\label{mgeneqn1}\\
&=\frac{(a; q)_\infty}{(c, x_1, x_2, \ldots, x_k)_\infty}
{_{k+1}\phi_{k}}\left({{c/a, x_1, x_2, \ldots, x_k}
\atop{0, 0, \ldots, 0}}; q, a\right).\nonumber
\end{align}
Andrews \cite {Andrews1972} has proved the following formula for the $q$-Lauricella function:
\begin{align*}
\sum_{n_1, n_2, \ldots, n_k=0}^\infty \frac{(a; q)_{n_1+n_2+\cdots+n_k}(b_1; q)_{n_1}(b_2; q)_{n_2}\cdots (b_k; q)_{n_k}
x_1^{n_1} x_2^{n_2}\cdots x_k^{n_k}}
{(c; q)_{n_1+n_2+\cdots+n_k}(q; q)_{n_1}(q; q)_{n_2}\cdots (q; q)_{n_k}}\\
=\frac{(a, b_1x_1, b_2x_2, \ldots, b_kx_k; q)_\infty}{(c, x_1, x_2, \ldots, x_k)_\infty}
{_{k+1}\phi_{k}}\left({{c/a, x_1, x_2, \ldots, x_k}
\atop{bx_1, bx_2, \ldots, bx_k}}; q, a\right).
\end{align*}
Taking $b_1=b_2=\cdots=b_k=0$ in the above equation, we conclude that
\begin{align*}
\sum_{n_1, n_2, \ldots, n_k=0}^\infty \frac{(a; q)_{n_1+n_2+\cdots+n_k}
x_1^{n_1} x_2^{n_2}\cdots x_k^{n_k}}
{(c; q)_{n_1+n_2+\cdots+n_k}(q; q)_{n_1}(q; q)_{n_2}\cdots (q; q)_{n_k}}\\
=\frac{(a; q)_\infty}{(c, x_1, x_2, \ldots, x_k)_\infty}
{_{k+1}\phi_{k}}\left({{c/a, x_1, x_2, \ldots, x_k}
\atop{0, 0, \ldots, 0}}; q, a\right).
\end{align*}
Equating the above equation and (\ref{mgeneqn1}), we find that
\[
\alpha_{n_1, n_2, \ldots, n_k}=\frac{(a; q)_{n_1+n_2+\cdots+n_k}}
{(c; q)_{n_1+n_2+\cdots+n_k}(q; q)_{n_1}(q; q)_{n_2}\cdots (q; q)_{n_k}}.
\]
\end{proof}
Substituting this into (\ref{mgeneqn}), we complete the proof of
Theorem~\ref{mgrogers}.
%%%%%%%%%%%%%%%%%%%%%%%%%%%%%%%%%%%%%%%%%%%%%%%%%%%%%%%%%%%%%%%%%%%%%%%%%%%%%%%%%%%%%%%%%%%%%%%%%%%%
\section{Some notes on the $q$-exponential operator identities }
%%%%%%%%%%%%%%%%%%%%%%%%%%%%%%%%%%%%%%%%%%%%%%%%%%%%%%%%%%%%%%%%%%%%%%%%%%%%%%%%%%%%%%%%%%%%%%%%%%%%
In this section, we will revisit the theory of $q$-exponential operator using Proposition~\ref{qppmotivation}.

Using the $q$-derivative operator, we can construct the $q$-exponential operator $T(y \mathcal{D}_{q, x})$ by
\begin{equation}
T(y \mathcal{D}_{q, x})=\sum_{n=0}^\infty \frac{(y \mathcal{D}_{q, x})^n }{(q; q)_n}.
\label{liueqn1}
\end{equation}
If we replace $q$ by $q^{-1}$ in the above equation, we obtain the dual operator of $T(y \mathcal{D}_{q, x})$
as follows
\begin{equation}
T(y \mathcal{D}_{q^{-1}, x})=\sum_{n=0}^\infty \frac{(-y \mathcal{D}_{q^{-1}, x})^n q^{n(n+1)/2}}{(q; q)_n}.
\label{liueqn2}
\end{equation}
\begin{defn}\label{opedefn} If $f(x)$ is analytic near $x=0,$ then, we define
\begin{align*}
T(y \mathcal{D}_{q, x})\{f(x)\}&=\sum_{n=0}^\infty \frac{y^n }{(q; q)_n}\mathcal{D}_{q, x}^n \{f(x)\},\\
T(y \mathcal{D}_{q^{-1}, x})\{f(x)\}&=\sum_{n=0}^\infty \frac{(-y)^n  q^{n(n+1)/2} }{(q; q)_n}\mathcal{D}_{q^{-1}, x}^n \{f(x)\}.
\end{align*}
\end{defn}
The following lemma is a modified version of \cite[Theorems~1 and 2]{Liu2010}, which is
the $q$-exponential operator form of Proposition~\ref{qppmotivation}.
This lemma indicates that the $q$-exponential operator method has a solid mathematical foundation.
\begin{lem}\label{liulemope} Suppose that $f(x, y)$ is a two-variable analytic function near $(x, y)=(0, 0).$
Then
\begin{itemize}
\item[(i)]
$f(x, y)=T(y \mathcal{D}_{q, x})\{f(x, 0)\},$ if and only if  $\partial_{q, x}\{f\}=\partial_{q, y}\{f\}$.
\item[(ii)]
$f(x, y)=T(y \mathcal{D}_{q^{-1}, x})\{f(x, 0)\},$ if and only if $\partial_{q^{-1}, x}\{f\}=\partial_{q^{-1}, y}\{f\}$.
\end{itemize}
\end{lem}
\begin{proof} We only prove (i). The proof of (ii) is similar to that of (i), so is omitted.
If $f(x, y)=T(y \mathcal{D}_{q, x})\{f(x, 0)\},$ then, by a direct computation, we find that
$\partial_{q, x}\{f\}=\partial_{q, y}\{f\}=T(y \mathcal{D}_{q, x})\mathcal{D}_{q, x}\{f(x, 0)\}.$
Conversely, if $f$ is a two-variable analytic function near $(x, y)=(0, 0),$
which satisfying $\partial_{q, x}\{f\}=\partial_{q, y}\{f\},$ then,  we may assume that
\[
f(x, y)=\sum_{n=0}^\infty A_n(x) y^n.
\]
Substituting the equation into the $q$-partial differential equation,
$\partial_{q, x}\{f\}=\partial_{q, y}\{f\}$, we easily find that
\[
A_n(x)=\frac{\mathcal{D}_{q, x}\{A_n(x)\}}{1-q^n}=\cdots=\frac{\mathcal{D}^n_{q, x}\{A_0(x)\}}{(q; q)_n}.
\]
It is easily seen that $A_0(x)=f(x, 0).$ Thus,  we find that
\[
f(x, y)=\sum_{n=0}^\infty \frac{\mathcal{D}^n_{q, x}\{f(x, 0)\}}{(q; q)_n}y^n
=T(y \mathcal{D}_{q, x})\{f(x, 0)\}.
\]
\end{proof}
\begin{prop} \label{cliupp}Suppose that $T(y \mathcal{D}_{q, x})$ and $T(y \mathcal{D}_{q^{-1}, x})$
are defined by $(\ref{liueqn1})$. Then
\begin{itemize}
\item[(i)] For $\max\{|xs|, |xt|, |ys|, |yt|\}<1$, we have the operator identity
\[
T(y \mathcal{D}_{q, x})\left\{\frac{1}{(xs, xt; q)_\infty}\right\}
=\frac{(xyst; q)_\infty}{(xs, xt, ys, yt; q)_\infty}.
\]
\item[(ii)] If $|xyst/q|<1,$ then, we have the operator identity
\[
T(y \mathcal{D}_{q^{-1}, x})\left\{(xs, xt; q)_\infty\right\}
=\frac{(xs, xt, ys, yt; q)_\infty}{(xyst/q; q)_\infty}.
\]
\end{itemize}
\end{prop}
Many applications of these two $q$-operator identities and their extensions  have appeared in the literature,
see, for example, \cite{Cao2009, Cao2010, CL1, CL2, Liu2003, Zhang, ZhangWang}.
The original proofs of these two operator identities in \cite{CL1, CL2} depend on the method of formal computation.
Next we use Lemma~\ref{liulemope} to give a very simple proof of  Proposition~\ref{cliupp}.
\begin{proof}  The proof of (ii) is similar to that of (ii), so we only prove (i).
If we use $f(x, y)$ to denote the right-hand side of the first equation in  Proposition~\ref{cliupp}.
Then it is easily seen that $f(x, y)$ is analytic near $(0, 0)$ and satisfies
\[
\partial_{q, x}\{f\}=\partial_{q, y}\{f\}
=\frac{s+t-(x+y)st}{1-xyst}f(x, y).
\]
Thus, by (i) in Proposition~\ref{liulemope}, we have $f(x, y)=T(y \mathcal{D}_{q, x})\{f(x, 0)\}.$
This completes the proof of Proposition~\ref{cliupp}.
\end{proof}
From Definition~\ref{opedefn}, by a direct computation, we find the $q$-exponential operator
representations of $h_n$ and $g_n$ as follows
\begin{equation}
T(y \mathcal{D}_{q, x})\{x^n\}=h_n(x, y|q), ~ T(y \mathcal{D}_{q^{-1}, x})\{x^n\}=g_n(x, y|q).
\label{liueqn3}
\end{equation}
\begin{prop}\label{liuopeppa} If $f(x)$ is analytic near $x=0,$ then $T(y \mathcal{D}_{q, x})\{f(x)\}$
is analytic near $(x, y)=(0, 0).$
\end{prop}
\begin{proof} Since $f(x)$ is analytic near $x=0,$ there exists a positive number $r<1$ (without loss of generality),
such that
\begin{equation}
f(x)=\sum_{k=0}^\infty a_k x^k,\quad |x|<r<1.
\label{liueqn4}
\end{equation}
From Definition~\ref{opedefn} and by a direct computation, we conclude that
\begin{align*}
T(y \mathcal{D}_{q, x})\{f(x)\}&=\sum_{n=0}^\infty \frac{y^n}{(q; q)_n} \sum_{k=n}^\infty \frac{(q; q)_k}{(q; q)_{k-n}}a_k x^{k-n}\\
&=\sum_{n=0}^\infty \frac{y^n}{(q; q)_n} \sum_{m=0}^\infty \frac{(q; q)_{n+m}a_{n+m}}{(q; q)_{m}} x^{m}.
\end{align*}
Since $f(x)$ is analytic in $|x|<r,$ we have, by the root test, that
$
\lim_{n\to \infty} \text{sup}|a_n|^{1/n}\le 1.
$
It is obvious that $\lim_{n\to \infty}|(q; q)_{n+m}|=|(q; q)_\infty|<\infty. $ It follows that
\[
\lim_{n\to \infty} \text{sup}|a_{n+m}(q; q)_{n+m}|^{1/(n+m)}\le 1.
\]
Thus, for $r$ in (\ref{liueqn4}), there exists a large integer $N$ such that $|a_{n+m}(q; q)_{n+m}|\le (1+r)^{m+n}$
for all $n\ge N$. It follows that
\begin{align*}
\Big|\sum_{n=N}^\infty \frac{y^n}{(q; q)_n} \sum_{m=0}^\infty \frac{(q; q)_{n+m}a_{n+m}}{(q; q)_{m}} x^{m}\Big|
\le \sum_{n=N}^\infty \frac{|(1+r)y|^n}{(q; q)_n}\sum_{m=0}^\infty \frac{(1+r)^{m}|x|^m}{(q; q)_{m}}
\end{align*}
We further assume that $|x|\le r/2$, then from the above equation, we deduce that
\begin{align*}
\Big|\sum_{n=N}^\infty \frac{y^n}{(q; q)_n} \sum_{m=0}^\infty \frac{(q; q)_{n+m}a_{n+m}}{(q; q)_{m}} x^{m}\Big|
\le \sum_{n=N}^\infty \frac{|(1+r)y|^n}{(q; q)_n}\sum_{m=0}^\infty \frac{(r(r+1)/2)^{m}}{(q; q)_{m}},
\end{align*}
which is a uniformly and absolutely  convergent series for $|y|<1/(1+r).$ Thus $T(y \mathcal{D}_{q, x})\{f(x)\}$
is analytic near $y=0.$ Thus, by Hartog's theorem, we know that $T(y \mathcal{D}_{q, x})\{f(x)\}$ is analytic near $(x, y)=(0, 0).$
\end{proof}
A deep result of the $q$-exponential operator is the following operator identity \cite[Eq.(3.1)]{Liu2010}.
\begin{prop}\label{liuopeppb}For $\max\{|as|, |at|, |au|, |bs|, |bt|, |bu|, |abstu/v|\}<1, $
we have
\begin{align*}
T(b \mathcal{D}_{q, a})\left\{\frac{(av; q)_\infty}{(as, at, au; q)_\infty}\right\}
&=\frac{(av, bv, abstu/v; q)_\infty}{(as, at, au, bs, bt, bu; q)_\infty}\\
&\quad \times{_3\phi_2}\left({{v/s, v/t, v/u}\atop{av, bv}}; q, \frac{abstu}{v}\right).
\end{align*}
\end{prop}
Now we will give a new proof of this identity.
\begin{proof} If we use $f(a, b)$ to denote the left-hand side of the equation in Proposition~\ref{liuopeppb},
then from Proposition~\ref{liuopeppa}, we know that $f(a, b)$ is analytic near $(a, b)=(0, 0).$ Using Lemma~\ref{liulemope},
we find that $\partial_{q, a}\{f\}=\partial_{q, b}\{f\}.$ Thus, there exists a sequence $\{\alpha_n\}$ independent of
$a$ and $b$ such that
\begin{equation}
f(a, b)=\sum_{n=0}^\infty \alpha_n h_n(a, b|q).
\label{liueqn5}
\end{equation}
Taking $b=0$ in the above equation and using $h_n(a, 0|q)=a^n$, we immediately deduce that
\begin{equation}
\sum_{n=0}^\infty \alpha_n a^n=\frac{(av; q)_\infty}{(as, at, au; q)_\infty}.
\label{liueqn6}
\end{equation}
Using the $q$-binomial theorem and the generating function for $h_n,$ we obtain
\[
\frac{(av; q)_\infty}{(as; q)_\infty}=\sum_{n=0}^\infty \frac{(v/s; q)_n (as)^n}{(q; q)_n}, \quad
\frac{1}{(at, au; q)_\infty}=\sum_{n=0}^\infty h_n(t, u|q) \frac{a^n}{(q; q)_n}.
\]
Substituting the above two equations into (\ref{liueqn6}) and equating the coefficients of $a^n,$
we deduce that
\[
\alpha_n=\frac{1}{(q; q)_n}\sum_{k=0}^n {n\brack k}_q (v/s; q)_k s^k h_{n-k}(t, u|q).
\]
Substituting the above equation into (\ref{liueqn6}), we arrive at
\[
f(a, b)=\sum_{n=0}^\infty \frac{h_n(a, b|q)}{(q; q)_n} \sum_{k=0}^n {n\brack k}_q (v/s; q)_k s^k h_{n-k}(t, u|q).
\]
Interchanging the order of the summation of the above equation and simplifying, we find that
\begin{equation}
f(a, b)=\sum_{k=0}^\infty \frac{s^k(v/s; q)_k}{(q; q)_k}
\sum_{m=0}^\infty \frac{h_{m+k}(a, b|q)h_m(t, u|q)}{(q; q)_m}.
\label{liueqn7}
\end{equation}
Using the Carlitz formula in Theorem~\ref{Carlitzthm}, we find that the inner summation of the above equation
equals
\[
\frac{(abtu; q)_\infty}{(at, au, bt, bu; q)_\infty}
\sum_{j=0}^k {k\brack j}_q \frac{b^j a^{k-j}(at, au; q)_j}{(abtu; q)_j}.
\]
Substituting the above equation into (\ref{liueqn7}) and then interchanging the order of the summation,
we conclude that
\[
f(a, b)=\frac{(av, abtu; q)_\infty}{(as, at, au, bt, bu; q)_\infty}
{_3\phi_2}\left({{v/s, at, au}\atop{av, abtu}}; q, bs\right).
\]
Using Sears' $_3\phi_2$ transformation formula (see, for example, \cite[Theorem~3]{Liu2003}), we have
\[
{_3\phi_2}\left({{v/s, at, au}\atop{av, abtu}}; q, bs\right)
=\frac{(abstu/v, bv; q)_\infty}{(abtu, bs; q)_\infty}
{_3\phi_2}\left({{v/s, v/t, v/u}\atop{av, bv}}; q, \frac{abstu}{v}\right).
\]
Combining the above two equations, we complete the proof of Proposition~\ref{liuopeppb}.
\end{proof}
In the same way, we can also obtain the following $q$-exponential operator identity \cite[Eq.(3.2)]{Liu2010}.
\begin{prop}\label{liuopeppc}For $\max\{|av|, |bv|, |abstu/qv|\}<1, $
we have
\begin{align*}
T(b \mathcal{D}_{q^{-1}, a})\left\{\frac{(as, at, au; q)_\infty}{(av; q)_\infty}\right\}
&=\frac{(as, at, au, bs, bt, bu; q)_\infty}{(av, bv, abstu/qv; q)_\infty}\\
&\quad \times{_3\phi_2}\left({{s/v, t/v, u/v}\atop{q/av, q/bv}}; q, q\right).
\end{align*}
\end{prop}
The following proposition reveals a relationship of $q$-integral and
$q$-exponential operator.
\begin{prop}\label{liuopeppd} If $T(y \mathcal{D}_{q, x})$ is defined by $(\ref{liueqn1})$, then, we have
\begin{align*}
&T(b \mathcal{D}_{q, a})\left\{\frac{(acst; q)_\infty}{(as, at, au; q)_\infty}\right\}\\
&=\frac{(cs, ct; q)_\infty}{(1-q)t (q, s/t, qt/s, au, bu; q)_\infty}\int_{s}^t \frac{(qx/s, qx/t, abux; q)_\infty}{(ax, bx, cx; q)_\infty} d_q x.
\end{align*}
\end{prop}
\begin{proof} If we use $f(a, b)$ to denote the right-hand side of the above equation, then, it is
to verify that $f(a, b)$ is analytic near $(a, b)=(0, 0).$ A direct computation shows that
$\partial_{q, a}\{f\}=\partial_{q, b}\{f\}$ equals
\[
\frac{(cs, ct; q)_\infty}{(1-q)t (q, s/t, qt/s, au, bu; q)_\infty}\int_{s}^t
\frac{(qx/s, qx/t, qabux; q)_\infty(x+u-(a+b)ux)}
{(ax, bx, cx; q)_\infty} d_q x.
\]
Using the Andrews-Askey integral in Proposition~\ref{AAintegralpp}, we easily find that
\begin{align*}
f(a, 0)&=
\frac{(cs, ct; q)_\infty}{(1-q)t (q, s/t, qt/s, au; q)_\infty}\int_{s}^t \frac{(qx/s, qx/t; q)_\infty}{(ax, cx; q)_\infty} d_q x\\
&=\frac{(acst, q)_\infty}{(au, as, at; q)_\infty}.
\end{align*}
Thus, using (i) in Lemma~\ref{liulemope}, we immediately conclude that
\[
f(a, b)=T(b\mathcal{D}_{q, a})\{f(a, 0)\}
=T(b \mathcal{D}_{q, a})\left\{\frac{(acst; q)_\infty}{(as, at, au; q)_\infty}\right\}.
\]
\end{proof}
Combining Proposition~\ref{liuopeppb} and Proposition~\ref{liuopeppd}, we can obtain the following
proposition, which is equivalent to \cite[Theorem~9]{Liu2010}.
\begin{prop}\label{liuopeppe} For $\max\{|as|, |bs|, |cs|, |at|, |bt|, |ct|, |abu/c|\}<1,$ we have
\begin{align*}
&\int_{s}^t \frac{(qx/s, qx/t, abux; q)_\infty}{(ax, bx, cx; q)_\infty} d_q x\\
&=\frac{(1-q)t(q, s/t, qst, acst, bcst, abu/c; q)_\infty}
{(as, bs, cs, at, bt, ct; q)_\infty}
{_3\phi_2}\left({{cs, ct, cst/u}\atop{acst, bcst}}; q, \frac{abu}{c}\right).
\end{align*}
\end{prop}
\begin{prop} \label{liuopeppf}Let $\{f_n(x)\}$ be a sequence of analytic functions near $x=0$,
such that the series
\[
\sum_{n=0}^\infty f_n(x)
\]
converges uniformly to an analytic function $f(x)$ near $x=0$,  and the series
\[
\sum_{n=0}^\infty T(yD_{q, x})\{f_n(x)\}
\]
converges uniformly to an analytic function $f(x, y)$ near $(x, y)=(0, 0)$. Then we have
$f(x, y)=T(yD_{q, x})\{f(x)\}$, or
\[
T(yD_{q, x})\left\{\sum_{n=0}^\infty f_n(x)\right\}=\sum_{n=0}^\infty T(yD_{q, x})\{f_n(x)\}.
\]
\end{prop}
\begin{proof} If we use $f_n(x, y)$ to denote $T(yD_{q, x})\{f_n(x)\}, $ then, by Lemma~\ref{liulemope},
we find that $\partial_{q, x}\{f_n(x, y)\}=\partial_{q, y}\{f_n(x, y)\}$. It follows that
$\partial_{q, x}\{f(x, y)\}=\partial_{q, y}\{f(x, y)\}.$ Thus, using Lemma~\ref{liulemope} again,
we have
\[
f(x, y)=T(yD_{q, x})\{f(x, 0)\}=T(yD_{q, x})\{f(x)\}.
\]
\end{proof}
If we replace $T(yD_{q, x})$  by $T(yD_{q^{-1}, x})$ in proposition~\ref{liuopeppf}, we
obtain the following proposition.
\begin{prop} \label{liuopeppg}Let $\{f_n(x)\}$ be a sequence of analytic functions near $x=0$,
such that the series
\[
\sum_{n=0}^\infty f_n(x)
\]
converges uniformly to an analytic function $f(x)$ near $x=0$,  and the series
\[
\sum_{n=0}^\infty T(yD_{q^{-1}, x})\{f_n(x)\}
\]
converges uniformly to an analytic function $f(x, y)$ near $(x, y)=(0, 0)$. Then we have
$f(x, y)=T(yD_{q^{-1}, x})\{f(x)\}$, or
\[
T(yD_{q^{-1}, x})\left\{\sum_{n=0}^\infty f_n(x)\right\}=\sum_{n=0}^\infty T(yD_{q^{-1}, x})\{f_n(x)\}.
\]
\end{prop}
Next we will use the above proposition to give an extension of $q$-Gauss summation. For this purpose,
we also need Tannery's Theorem (see, for example, \cite{Boas}).
\begin{thm}\label{tanthm} For each nonnegative integer  $n$, let $s(n)=\sum_{k>o} f_k(n)$ is a finite sum
or a convergent series.  If $\lim_{n\to \infty} f_k(n)=f_k$, $|f_k(n)|\le M_k$, and $\sum_{k> 0}M_k<\infty,$
then
\[
\lim_{n\to \infty} s(n)=\sum_{k=0}^\infty f_k.
\]
\end{thm}
\begin{prop}\label{liuopeppgauss} For $\max\{|x|, |abxy|\}<1,$ we have the summation
\begin{align*}
\sum_{n=0}^\infty \frac{(a, b; q)_n x^n}{(q, abxy; q)_n}
{_3\phi_2}\({{q^{-n}, 1/x, 1/y}\atop{a, b}}; q, abxyq^n\)
=\frac{(ax, bx; q)_\infty}{(x, abxy; q)_\infty}.
\end{align*}
\end{prop}
When $y=1,$ the above proposition immediately reduces to the $q$-Gauss summation.
\begin{proof}
If we denote the $_3\phi_2$ series in Proposition~\ref{liuopeppgauss} by $f_k(n), $ then,  by a simple computation,
we have
\begin{align*}
&f_k(n)
=(-1)^k {n\brack k}_q \frac{(1/x, 1/y ;q)_k}{(q, a, b; q)_k }q^{k(k-1)/2}(abxy)^k,\\
&\lim_{n\to \infty}f_k(n)=(-1)^k  \frac{(1/x, 1/y ;q)_k}{(q, a, b; q)_k }q^{k(k-1)/2}(abxy)^k,\\
&\quad|f_k(n)| \le \frac{|(1/x, 1/y ;q)_k|}{(q;q)_k |(a, b; q)_k|} |abxy|^k=M_k.
\end{align*}
Using the ratio test, we find, for $|abxy|<1,$ that $\sum_{k \ge 0} M_k<\infty.$
Thus, by Tannery's Theorem, we deduce that for $|abxy|<1,$
\[
\lim_{n\to \infty} {_3\phi_2}\left({{q^{-n}, 1/x, 1/y}\atop{a, b}}; q, abxyq^n \right)
=\sum_{k=0}^\infty (-1)^k  \frac{(1/x, 1/y ;q)_k}{(q, a, b; q)_k }q^{k(k-1)/2}(abxy)^k.
\]
Hence, using the ratio test, we find the series in  Proposition~\ref{liuopeppgauss} converges uniformly
to an analytic function $f(a, b)$ for $\max\{|x|, |abxy|\}<1.$

Rewrite the $q$-binomial theorem in the form
\[
\sum_{n=0}^\infty \frac{x^n}{(q; q)_n} \frac{(a; q)_\infty}{(aq^n, ax; q)_\infty}
=\frac{1}{(x; q)_\infty}.
\]
Multiplying both sides of the above equation by $1/(ay; q)_\infty,$ we deduce that
\begin{equation}
\sum_{n=0}^\infty \frac{x^n}{(q; q)_n} \frac{(a; q)_\infty}{(aq^n, ax, ay; q)_\infty}
=\frac{1}{(x; q)_\infty(ay; q)_\infty}.
\label{qgauss1}
\end{equation}
Using Proposition~\ref{liuopeppb}, we have,  for $\max\{|ax|, |ay|, |a|, |abxy|\}<1,$  that
\begin{align*}
&\sum_{n=0}^\infty \frac{x^n}{(q; q)_n} T(b \mathcal{D}_{q, a})\left\{\frac{(a; q)_\infty}{(aq^n, ax, ay; q)_\infty}\right\}\\
&=\frac{( abxy; q)_\infty}{( ax, ay, bx, by; q)_\infty}
\sum_{n=0}^\infty \frac{(a, b; q)_n x^n}{(q, abxy; q)_n}
{_3\phi_2}\({{q^{-n}, 1/x, 1/y}\atop{a, b}}; q, abxyq^n\).
\end{align*}
We have proved that the right-hand side of the above equation is analytic near $(a, b)=(0, 0).$ Thus, we find
that the right-hand side of the above equation equals
\begin{align*}
T(b \mathcal{D}_{q, a})\left\{\sum_{n=0}^\infty \frac{x^n}{(q; q)_n} \frac{(a; q)_\infty}{(aq^n, ax, ay; q)_\infty}\right\}
&=\frac{1}{(x; q)_\infty} T(b \mathcal{D}_{q, a})\left\{\frac{1}{(ay; q)_\infty}\right\}\\
&=\frac{1}{(x, ay, by; q)_\infty}.
\end{align*}
Combining the above two equations, we complete the proof of Proposition~\ref{liuopeppgauss}.
\end{proof}
In the same way, we can prove the following transformation for $q$-series, which reduces
to the Jackson q-analogue of the Euler transformation when $b=y.$
\begin{prop}\label{qeulerpp} For $\max\{|b|, |x|, |abxy|\}<1,$ we have
\begin{align*}\\
&\frac{(b; q)_\infty}{(x; q)_\infty} {_2\phi_1}\left({{ax, cx}\atop{acxy}}; q, b\right)\\
&=\sum_{n=0}^\infty \frac{(ab, bc; q)_n x^n}{(q, acxy; q)_n}
{_3\phi_2}\left({{q^{-n}, b/x, b/y}\atop{ab, bc}}; q, abxyq^n\right).
\end{align*}
\end{prop}

We end this section with an extension of the Sears $_4\phi_3$ transformation formula.
\begin{prop}\label{liuopepph} We have the $q$-transformation formula
\begin{align*}
&\sum_{j=0}^n \frac{(q^{-n}, be^{i\theta}, be^{-i\theta}, abcduq^{n-1}/v; q)_j q^j}
{(q, ab, bc, bd; q)_j}{_3\phi_2}\left({{bq^{j}/v, cq^{n}/v, u/v}\atop{q/av, q/dv}}; q, q\right)\\
&=\frac{(ac, cd; q)_n }{(ab, bd; q)_n }\left(\frac{b}{c}\right)^n\sum_{j=0}^n \frac{(q^{-n}, ce^{i\theta}, ce^{-i\theta}, abcduq^{n-1}/v; q)_j q^j}
{(q, ac, bc, cd; q)_j}{_3\phi_2}\left({{cq^{j}/v, bq^{n}/v, u/v}\atop{q/av, q/dv}}; q, q\right).
\end{align*}
\end{prop}
If we let $u=v$ in the above equation and use that obvious fact that $(1; q)_0=1$ and $(1; q)_k=0$ for $k\ge 1,$
we find that Proposition~\ref{liuopepph} reduces to the Sears $_4\phi_3$ transformation \cite[Theorem~4]{Liu2003}.
\begin{proof} Recall the Sears $_3\phi_2$ transformation formula (see, for example, \cite[p. 123]{Liu2003}), which states
\begin{align*}
&{_3\phi_2}\({{a_1, a_2, a_3}\atop{b_1, b_2}}; q, \frac{b_1b_2}{a_1a_2a_3}\)\\
&=\frac{(b_2/b_3, b_1b_2/a_1a_2; q)_\infty}
{(b_2, b_1b_2/a_1a_2a_3; q)_\infty} {_3\phi_2}\({{b_1/a_1, b_1/a_2, a_3}\atop{b_1, b_1b_2/a_1a_2}}; q, \frac{b_2}{a_3}\).
\end{align*}
Taking $a_3=q^{-n}$ in the above equation, replacing $(a_1, a_2, b_1, b_2, q)$ by
\[
(1/a_1, 1/a_2, 1/b_1, 1/b_2, 1/q)
\]
and performing some calculation, we obtain the transformation formula
\begin{align*}
{_3\phi_2}\({{q^{-n}, a_1, a_2}\atop{b_1, b_2}}; q, q\)
=\frac{(b_1b_2/a_1a_2; q)_n}{(b_2; q)_n} \(\frac{a_1a_2}{b_1}\)^n
{_3\phi_2}\({{q^{-n}, \frac{b_1}{a_1},\frac {b_1}{a_2}}\atop{b_1, \frac{b_1b_2}{a_1a_2}}}; q, q\).
\end{align*}
Replacing $(a_1, a_2, b_1, b_2)$ by $(be^{i\theta}, be^{-i\theta}, bc, ab)$ in the above
equation and making some calculation,  we find that
\begin{align*}
&\sum_{j=0}^n \frac{(q^{-n}, be^{i\theta}, be^{-i\theta}; q )_j q^j}{(q, bc; q)_j}(abq^j, acq^n; q)_\infty\\
&=\(\frac{b}{c}\)^n \sum_{j=0}^n \frac{(q^{-n}, ce^{i\theta}, ce^{-i\theta}; q )_j q^j}{(q, bc; q)_j}(acq^j, abq^n; q)_\infty.
\end{align*}
Multiplying both sides of the above equation by $(au; q)_\infty/(av; q)_\infty,$
and then taking the action of $T(d D_{q^{-1}, a})$ on both sides of the resulting equation, we deduce that
\begin{align*}
&\sum_{j=0}^n \frac{(q^{-n}, be^{i\theta}, be^{-i\theta}; q )_j q^j}{(q, bc; q)_j}T(d D_{q^{-1}, a})\left\{\frac{(abq^j, acq^n, au; q)_\infty}{(av; q)_\infty}\right\}\\
&=\(\frac{b}{c}\)^n \sum_{j=0}^n \frac{(q^{-n}, ce^{i\theta}, ce^{-i\theta}; q )_j q^j}{(q, bc; q)_j}
T(d D_{q^{-1}, a})\left\{\frac{(acq^j, abq^n, au; q)_\infty}{(av; q)_\infty}\right\}.
\end{align*}
Using the $q$-exponential operator in Proposition~\ref{liuopeppc}, we immediately find that
\begin{align*}
T(d D_{q^{-1}, a})\left\{\frac{(abq^j, acq^n, au; q)_\infty}{(av; q)_\infty}\right\}
&=\frac{(abq^j, acq^n, au, bdq^j, cdq^n, du; q)_\infty}{(av, dv, abcduv^{-1}q^{n+j-1}; q)_\infty}\\
&\quad\times  {_3\phi_2}\({{q^{-n}, bq^j/v, cq^n/v, u/v}\atop{q/av, q/dv}}; q, q\),
\end{align*}
\begin{align*}
T(d D_{q^{-1}, a})\left\{\frac{(acq^j, abq^n, au; q)_\infty}{(av; q)_\infty}\right\}
&=\frac{(acq^j, abq^n, au, cdq^j, bdq^n, du; q)_\infty}{(av, dv, abcduv^{-1}q^{n+j-1}; q)_\infty}\\
&\quad\times  {_3\phi_2}\({{q^{-n}, cq^j/v, bq^n/v, u/v}\atop{q/av, q/dv}}; q, q\).
\end{align*}
Combining the above three equations, we complete the proof of Proposition~\ref{liuopepph}.
\end{proof}

\section{Acknowledgments}
I am  grateful to the referee for many very helpful comments.
The  author was supported in part by
the National Science Foundation of China.

\end{document}